\documentclass[11pt,reqno]{amsart}

\usepackage{amsfonts, amsthm, amsmath}
\allowdisplaybreaks[4]

\usepackage{rotating}

\usepackage{tikz}

\usepackage{graphics}

\usepackage{amssymb}

\usepackage{amscd}

\usepackage[latin2]{inputenc}

\usepackage{t1enc}

\usepackage[mathscr]{eucal}

\usepackage{indentfirst}

\usepackage{graphicx}

\usepackage{graphics}

\usepackage{pict2e}

\usepackage{mathrsfs}

\usepackage{enumerate}
\usepackage{cite}
\usepackage{color}
\usepackage{epic}
\usepackage{hyperref} 
\usepackage{framed}
\usepackage{booktabs} 
\usepackage{mathdots} 

\numberwithin{equation}{section}
\def\blue{\textcolor{blue}}
\def\red{\textcolor{red}}
\theoremstyle{plain}

\newtheorem{theorem}{Theorem}[section]

\newtheorem{lemma}[theorem]{Lemma}

\newtheorem{corollary}[theorem]{Corollary}

\newtheorem{proposition}[theorem]{Proposition}

\newtheorem{Fact}[theorem]{Fact}

\theoremstyle{definition}

\newtheorem{Def}[theorem]{Definition}

\newtheorem{example}[theorem]{Example}

\newtheorem{conj}[theorem]{Conjecture}

\newtheorem{remark}[theorem]{Remark}

\renewcommand{\S}{\mathfrak{S}}
\def\A{\mathfrak{A}}
\def\des{\mathrm{des}}

\def\exc{\mathrm{exc}}
\def\wexc{\mathrm{wexc}}
\def\iexc{\mathrm{iexc}}
\def\fix{\mathrm{fix}}
\def\Fix{\mathrm{Fix}}
\def\per{\mathrm{per}}
\def\sgn{\mathrm{sgn}}
\def\DES{\mathrm{DES}}
\def\exph{\mathrm{exph}}

\def\Dasc{\mathrm{Dasc}}
\def\Ddes{\mathrm{Ddes}}
\def\grn{\mathrm{grn}}
\def\sign{{\rm sign}}

\def\Dest{\mathrm{Dest}}
\def\Exct{\mathrm{Exct}}

\def\Lag{\mathfrak{L}}
\def\pp{\operatorname{(2--13)}}
\def\Peak{\mathrm{Peak}}
\def\Val{\mathrm{Val}}
\def\Orb{\mathrm{Orb}}

\def\Im{\mathrm{Im}}

\newcommand{\N}{\mathbb{N}}
\newcommand{\Z}{\mathbb{Z}}

\newcommand{\D}{\mathfrak{D}}

\def\l{\left}
\def\r{\right}
\def\bg{\bigg}
\def\({\bg(}
\def\){\bg)}
\def\t{\text}
\def\f{\frac}

\def\per{{\rm per}}

\def\ls{\leqslant}
\def\gs{\geqslant}

\def\eq{\equiv}

\begin{document}
\hbox{Polished version for publication}
\medskip

\title{Permanent identities, combinatorial sequences, and permutation statistics}

\author[S. Fu]{Shishuo Fu}
\address[Shishuo Fu]{College of Mathematics and Statistics, Chongqing University, Huxi campus, Chongqing 401331, P.R. China}
\email{fsshuo@cqu.edu.cn}

\author[Z. Lin]{Zhicong Lin}
\address[Zhicong Lin]{Research Center for Mathematics and Interdisciplinary Sciences, Shandong University \& Frontiers Science Center for Nonlinear Expectations, Ministry of Education, Qingdao 266237, P.R. China}
\email{linz@sdu.edu.cn}

\author[Z.-W. Sun]{Zhi-Wei Sun}
\address[Zhi-Wei Sun, corresponding author]{School of Mathematics, Nanjing University, Nanjing 210093, P.R. China}
\email{zwsun@nju.edu.cn}


\begin{abstract}
In this paper, we confirm six conjectures on the exact values of some permanents, relating them to the Genocchi numbers of the first and second kinds as well as the Euler numbers. For example, we prove that
$$\mathrm{per}\left[\left\lfloor\frac{2j-k}{n}\right\rfloor\right]_{1\le j,k\le n}=2(2^{n+1}-1)B_{n+1},$$
where $B_0,B_1,B_2,\ldots$ are the Bernoulli numbers.
We also show that
$$
\per\left[\sgn\left(\cos\pi\frac{i+j}{n+1}\right)\right]_{1\le i,j\le n}=\begin{cases}
-\sum_{k=0}^m\binom{m}{k}E_{2k+1}&\quad\text{if $n=2m+1$},\\
\sum_{k=0}^m\binom{m}{k}E_{2k}&\quad\text{if $n=2m$},
\end{cases}
$$
where $\sgn(x)$ is the sign function, and $E_0,E_1,E_2,\ldots$ are the Euler (zigzag) numbers.

 In the course of linking the evaluation of these permanents to the aforementioned combinatorial sequences, the classical permutation statistic -- the excedance number, together with several kinds of its variants, plays a central role. Our approach features recurrence relations, bijections, as well as certain elementary operations on matrices that preserve their permanents. Moreover, our proof of the second permanent identity leads to  a proof of Bala's conjectural continued fraction formula, and an unexpected permutation interpretation for the $\gamma$-coefficients of the $2$-Eulerian polynomials.

\end{abstract}


\maketitle
\section{Introduction}

For a matrix $A=[a_{i,j}]_{1\le i,j\le n}$ over a commutative ring with identity, its determinant and permanent are defined by
$$\det(A)=\det[a_{i,j}]_{1\le i,j\le n}=\sum_{\pi\in\S_n}\sign(\pi)\prod_{i=1}^n a_{i,\pi(i)}$$
and
$$\per(A)=\per[a_{i,j}]_{1\le i,j\le n}=\sum_{\pi\in\S_n}\prod_{i=1}^n a_{i,\pi(i)}$$
respectively, where $\S_n$ is the symmetric group of all permutations of $[n]:=\{1,\ldots,n\}$.
Determinants are widely used in mathematics, and permanents are useful in combinatorics.

It is well known that
$$\det[i^{j-1}]_{1\le i,j\le n}=\prod_{1\le i<j\le n}(j-i)=1!2!\cdots(n-1)!$$
as this is of Vandermonde's type. In contrast, Z.-W. Sun \cite{S18} proved that
$$\per[i^{j-1}]_{1\le i,j\le n}\eq0\pmod n\quad\t{for all}\ n\ge 3.$$

Z.-W. Sun \cite{S19} evaluated some determinants involving trigonometric functions.
For example, he proved that for any odd integer $n>1$ and integers $a$ and $b$ with $\gcd(ab,n)=1$, the following identity holds:
$$\det\l[\tan\pi\f{aj+bk}n\r]_{1\ls j,k\ls n-1}=\l(\f{-ab}n\r)n^{n-2},$$
where $(\f{\cdot}n)$ is the Jacobi symbol.

Motivated by the above work, Sun \cite{Sun} investigated arithmetic properties of some permanents.
For example, he showed that
\begin{equation}\label{(j+k)/n}
\per\l[\l\lfloor\f{j+k}n\r\rfloor\r]_{1\ls j,k\ls n}=2^{n-1}+1
\ \ \t{and}\ \ \per\l[\l\lfloor\f{j+k-1}n\r\rfloor\r]_{1\ls j,k\ls n}=1,
\end{equation}
where $\lfloor\cdot\rfloor$ is the floor function. Sun \cite{Sun} also proved that
$$\per\l[\tan\pi\f{j+k}p\r]_{1\le j,k\le p-1}\eq(-1)^{(p+1)/2}2p\pmod{p^2}$$
for any odd prime $p$.

For any positive integer $n$, clearly
$$\l\lfloor\f{2j-k}n\r\rfloor\in\{0,\pm1\}\quad\t{for all}\ j,k=1,\ldots,n.$$
Inspired by this as well as his preprint \cite{Sun},
Sun \cite{Sun-quest} posed the following novel conjecture involving the Bernoulli numbers $B_0,B_1,\ldots$ given by
$$\frac{x}{e^x-1}=\sum_{n=0}^{\infty}B_n\frac{x^n}{n!}\quad (|x|<2\pi).$$

\begin{conj}[Z.-W. Sun]\label{conj:Sun}
For any positive integer $n$, we have
\begin{align}
\per\left[\left\lfloor\frac{2j-k}{n}\right\rfloor\right]_{1\le j,k\le n}=2(2^{n+1}-1)B_{n+1}.
\end{align}
\end{conj}

This conjecture has aroused quite some interests and triggered further conjectures over the on-line forum {\tt MathOverflow} concerning the evaluations of various permanents. We aim to prove six of the conjectures (including the above one) posted there.

Three combinatorial sequences play major roles in the current work. They are the Genocchi numbers, the median Genocchi numbers, and the Euler numbers. We collect their definitions here and state the remaining five conjectures afterwards.

The Genocchi numbers (of the first kind) $G_1,G_2,\ldots$ are given by
$$\f{2x}{e^x+1}=\sum_{n=1}^\infty G_n\f{x^n}{n!}\ \ (|x|<\pi).$$
It is known that $G_n=2(1-2^n)B_n$ for any positive integer $n$ (cf. \cite[A036968]{oeis});
in particular, $G_{2n+1}=2(1-2^{2n+1})B_{2n+1}=0$
and $$(-1)^nG_{2n}=2(2^{2n}-1)(-1)^{n-1}B_{2n}>0$$
for all $n=1,2,3,\ldots$. Note that
$$G_1=1,\ G_2=-1,\ G_4=1,\ G_6=-3,\ G_8=17,\ G_{10}=-155,\ G_{12}=2073.$$

The median Genocchi numbers (or Genocchi numbers of the second kind, cf. \cite[A005439]{oeis}) $H_1,H_3,H_5,\ldots$ can be defined by their relation with $G_{2n}\ (n\ge1)$:
$$H_{2n-1}=(-1)^n\sum_{j=0}^{\lfloor\frac{n-1}{2}\rfloor}\binom{n}{2j+1}G_{2n-2j}\ \ \text{ for all }n=1,2,3,\ldots.$$
For example,
$$H_1=1,\ H_3=2,\ H_5=8,\ H_7=56,\ H_9=608,\ H_{11}=9440.$$

The {\em Euler numbers} $(E_n)_{n\geq0}$  are defined as the coefficients of the Taylor expansion
$$
\sec(x)+\tan(x)=\sum_{n\geq0} E_n\frac{x^n}{n!}=1+1\frac{x}{1!}+1\frac{x^2}{2!}+2\frac{x^3}{3!}+5\frac{x^4}{4!}+16\frac{x^5}{5!}+ 61\frac{x^6}{6!}+272\frac{x^7}{7!}+\cdots.
$$
It was Andr\'e~\cite{Andre} in $1879$ who first discovered the interpretation of $E_n$ as the number of {\em alternating (down-up) permutations} of length $n$. The Euler numbers $E_{2n}$ of even indices are called {\em secant numbers}, while those $E_{2n-1}$ with odd indices are called {\em tangent numbers}.

Motivated by Conjecture \ref{conj:Sun}, P. Luschny~\cite[A005439]{oeis} made the following similar conjecture involving median Genocchi numbers.

\begin{conj}[P. Luschny]\label{conj:Lus} Let $n$ be any positive integer and define
$$M_{2n}:=\left[\left\lfloor\frac{2j-k-1}{2n}\right\rfloor\right]_{1\le j,k\le 2n}.$$
 Then we have
\begin{align}
\per(M_{2n})=(-1)^nH_{2n-1}.
\end{align}
\end{conj}

Recall that the sign function is given by
$$\sgn(x)=\begin{cases}
1&\t{if}\ x>0,\\
0&\t{if}\ x=0,\\-1&\t{if}\ x<0.
\end{cases}$$

The following third, fourth, and fifth conjectures (involving the sign function) were
raised by Deyi Chen~\cite{Che1,Che2}.

\begin{conj}[D. Chen]\label{conj:ChenA} Let $n$ be any positive integer, and set
$$A_{2n}:=\l[\sgn\l(\tan\pi\frac{i+j}{2n+1}\r)\r]_{1\le i,j\le 2n}.$$
Then we have
\begin{align}
\per(A_{2n})=\per(A^{-1}_{2n})=(-1)^nE_{2n}.
\end{align}
\end{conj}

\begin{conj}[D. Chen]\label{conj:Chen} Let $n$ be a positive integer, and set
$$P_n:=\left[\sgn\left(\sin\pi\frac{i+j}{n+1}\right)\right]_{1\le i,j\le n}.$$
Then
\begin{equation}\label{P2n}
\per(P_{2n})=\per(P^{-1}_{2n})=(-1)^nE_{2n}.
\end{equation}
\end{conj}

\begin{conj}[D. Chen]\label{conj:Chen2}
Let $n$ be a positive integer, and set
 $$Q_n:=\left[\sgn\left(\sin\pi\frac{i+2j}{n+1}\right)\right]_{1\le i,j\le n}.$$  Then
\begin{equation}
\per(Q_n)=(-1)^nE_{n}.
\end{equation}
\end{conj}

In view of the above conjectures relating permanents involving ``$\tan$'' and ``$\sin$'' directly to (signed) Euler numbers, it seems natural to consider the trigonometric function ``$\cos$'' instead, and compute the corresponding permanents. This consideration leads to the most recent conjecture of Deyi Chen~\cite{Che3}.

\begin{conj}[D. Chen]\label{conj:Chen3} Let $n$ be a positive integer, and define
$$R_n:=\left[\sgn\l(\cos\pi\frac{i+j}{n+1}\r)\right]_{1\le i,j\le n}.$$ Then
\begin{equation}\label{bino:chen}
\per(R_n)=\begin{cases}
-\sum_{k=0}^m\binom{m}{k}E_{2k+1}&\quad\text{if $n=2m+1$},\\
\sum_{k=0}^m\binom{m}{k}E_{2k}&\quad\text{if $n=2m$}.
\end{cases}
\end{equation}
\end{conj}

We are going to prove Conjectures \ref{conj:Sun} and \ref{conj:Lus} in Section \ref{sec:conj of sun}, where four permutation interpretations of Kreweras' triangle are derived as byproduct. Relying on the classical sign-balance results for the excedance polynomials over permutations and derangements, as well as certain elementary action on matrices, we will prove Conjectures \ref{conj:ChenA}--\ref{conj:Chen2} in Section \ref{sec:conj of chen}. The proof of Conjecture \ref{conj:Chen3} is a combination of the Foata--Strehl action~\cite{FSt} and the bivariate generating functions of two types of Poupard numbers studied by Foata and Han \cite{FH}. This will be done in Section~\ref{sec:conj of chen2}, where a conjecture posted to \cite[A005799]{oeis} by Peter Bala is also confirmed. Moreover, we will give in the last section a new permutation interpretation for the $\gamma$-coefficients of the descent polynomials on the multiset $\{1,1,2,2,\ldots,n,n\}$, and conclude our paper by posing some related open problems for further research.

\section{Proofs of Conjectures~\ref{conj:Sun}-\ref{conj:Lus} and some relevant results}\label{sec:conj of sun}

\subsection{Proof of Conjecture~\ref{conj:Sun}}
We begin with some initial observations and analysis on Conjecture~\ref{conj:Sun}. Let $$L_n:=\left[\left\lfloor\frac{2j-k}{n}\right\rfloor\right]_{1\le j,k\le n},$$ which is a matrix with entries in $\{0,\pm1\}$. We observe the following sign patterns (whether an entry is $1$, $0$ or $-1$) of the matrix $L_n$.

\begin{Fact}\label{fact}
\begin{enumerate}[{\rm (1)}]
\item For even $n=2m\ge 2$, we have $\per(L_n)=B_{n+1}=0$, since the $m$-th row of $L_n$ contains only zeros.
\item For odd $n=2m+1$, the first $m$ rows of $L_n$ begin with $0$s and end with $-1$s, the next $m$ rows (the ($m+1$)-th row to the $2m$-th row) begin with $1$s and end with $0$s, while the last row contains only $1$s. In particular, we have $\sgn(\per(L_n))=(-1)^m=\sgn(B_{n+1})=-\sgn(G_{2m+2})$.
\item For odd $n=2m+1$, the $m$-th row of $L_n$ begins with $2m$ consecutive $0$s and ends with one $-1$, while the $(m+1)$-th row of $L_n$ begins with one $1$ and continues with $2m$ consecutive $0$s.
\end{enumerate}
\end{Fact}

In view of Fact~\ref{fact}(1)-(2), it suffices to show that
\begin{align}
\label{id:per=G}
|\per(L_{2m+1})|=(-1)^{m+1}G_{2m+2}\  \text{ for every $m=0,1,2,\ldots$.}
\end{align}

Now Fact~\ref{fact}(3) tells us that when we expand $\per(L_{2m+1})$ to get non-zero terms, the choices for the $m$-th and $(m+1)$-th rows are unique, namely, we are forced to choose the last entry $-1$ for the $m$-th row and the first entry $1$ for the $(m+1)$-th row. Therefore, we shall delete these two rows, as well as the first and last columns of $L_{2m+1}$, and extract $-1$ from each of the first $m-1$ rows to consider the matrix
$\tilde{L}_{2m-1}=(\ell_{i,j})_{1\le i,j\le 2m-1}$, where
$$\ell_{i,j}=\begin{cases}
1 & \text{if $1\le i\le m-1\ \&\ 2i\le j\le 2m-1$,}\\
 & \text{or $m\le i\le 2m-2\ \&\ 1\le j\le 2(i-m)+2$, or $i=2m-1$,}\\
0 & \text{otherwise,}
\end{cases}$$
and then observe that
\begin{align}
\label{id:pertildeA=perA}
\per(\tilde{L}_{2m-1})=|\per(L_{2m+1})|.
\end{align}
For example, the first four matrices are $\tilde{L}_1=[1]$, and
\begin{align*}
\tilde{L}_3=\begin{bmatrix}
0 & 1 & 1 \\
1 & 1 & 0 \\
1 & 1 & 1
\end{bmatrix},
\;
\tilde{L}_5=\begin{bmatrix}
0 & 1 & 1 & 1 & 1\\
0 & 0 & 0 & 1 & 1\\
1 & 1 & 0 & 0 & 0\\
1 & 1 & 1 & 1 & 0\\
1 & 1 & 1 & 1 & 1
\end{bmatrix},
\;
\tilde{L}_7=\begin{bmatrix}
0 & 1 & 1 & 1 & 1 & 1 & 1\\
0 & 0 & 0 & 1 & 1 & 1 & 1\\
0 & 0 & 0 & 0 & 0 & 1 & 1\\
1 & 1 & 0 & 0 & 0 & 0 & 0\\
1 & 1 & 1 & 1 & 0 & 0 & 0\\
1 & 1 & 1 & 1 & 1 & 1 & 0\\
1 & 1 & 1 & 1 & 1 & 1 & 1
\end{bmatrix}.
\end{align*}

In order to calculate $\per(\tilde{L}_{2m-1})$, we expand $\tilde{L}_{2m-1}$ along the bottom row, and denote the $(2m-1,i)$-minor by $\tilde{L}_{2m-1,i}$ ($1\le i\le 2m-1$). That is,  $\tilde{L}_{2m-1,i}$ is the permanent of the submatrix obtained from deleting the $(2m-1)$-th row and the $i$-th column of $\tilde{L}_{2m-1}$. The following recurrence relation fully characterizes these minors, and it is vital to our proof of Conjecture~\ref{conj:Sun}.

\begin{lemma}\label{rec:krew}
For integers $m\gs k\gs1$, we have
\begin{equation}
\label{rec+}
\tilde{L}_{2m-1,2k-1} =\tilde{L}_{2m-1,2k-2}+\sum_{i=2k-2}^{2m-3}\tilde{L}_{2m-3,i}
\end{equation}
and
\begin{equation}\label{rec-}
\tilde{L}_{2m-1,2k} =\tilde{L}_{2m-1,2k-1}-\sum_{i=1}^{2k-2}\tilde{L}_{2m-3,i},
\end{equation}
where we set $\tilde{L}_{2m-1,0}=\tilde{L}_{2m-1,2m}=0$.
\end{lemma}
\begin{proof}
If we delete the $m$-th row and the $(2m-1)$-th row, and the first two columns of $\tilde{L}_{2m-1}$, we get a submatrix that becomes $\tilde{L}_{2m-3}$ once we flip it upside down, then left to right. This shows the $k=1$ case of \eqref{rec+}.

Next, for $k\ge 2$, comparing the $(2k-1)$-th column of $\tilde{L}_{2m-1}$ with the $(2k-2)$-th column, we see that the only discrepancy is $\ell_{m+k-2,2k-2}=1$ while $\ell_{m+k-2,2k-1}=0$. This means the difference of the two minors, $\tilde{L}_{2m-1,2k-1} -\tilde{L}_{2m-1,2k-2}$, is given by the permanent of the submatrix obtained from deleting the $(m+k-2)$-th, $(2m-1)$-th rows, and the $(2k-2)$-th, $(2k-1)$-th columns of $\tilde{L}_{2m-1}$. Expanding this permanent along the $(k-1)$-th row, we see that it coincides with the summation
$$\sum_{i=2k-2}^{2m-3}\tilde{L}_{2m-3,i}.$$
This proves \eqref{rec+}. The equality \eqref{rec-} can be proved similarly, and we omit the details.
\end{proof}

The signless Genocchi numbers $(-1)^nG_{2n}$ possess   many interesting combinatorial and arithmetic properties. The first combinatorial interpretation of Genocchi numbers was found by Dumont~\cite{du} in 1974, which asserts that
$|\mathcal{D}_{2n-1}|=(-1)^nG_{2n}$ with
 $$
 \mathcal{D}_{2n-1}:=\{\sigma\in\S_{2n-1}:\forall\, i\in[2n-2], \text{ $\sigma(i)>\sigma(i+1)$ if and only if $\sigma(i)$ is even}\}.
 $$
 Since then, many other interpretations of Genocchi numbers have been found in the literature; see~\cite{Big,Bur,Fei,Han,Kre,LY} and the references therein.

Our final step for the proof of Conjecture~\ref{conj:Sun} is to realize that the recurrences \eqref{rec+} and \eqref{rec-}  are precisely the recurrences for Kreweras' triangle~\cite{Kre}:
\begin{center}
\begin{tabular}{ccccccccccccccccc}
 & & & &  & &  & & 1 & &  & & & & & & \\
 & & & &  & & 1 & & 1 & & 1 & & & & & & \\
 & & & & 3 & & 3 & & 5 & & 3 & & 3 & & & & \\
 & & 17 & & 17 & & 31 & & 25 & & 31 & & 17 & & 17 & & \\
155 & & 155 & & 293 & & 259 & & 349 & & 259 & & 293 & & 155 & & 155 \\
 & & & &  & &  & & $\vdots$ & &  & & & & & &
\end{tabular}
\end{center}
This is a well-known triangle that refines the Genocchi numbers.
Its entry in the $m$-th row and the $k$-th column is given by
$$K_{2m-1,k}:=
|\{\sigma\in\mathcal{D}_{2m+1}:\, \sigma(1)=k+1\}|.
$$
For instance, $K_{5,2}=3$ counts the three qualified permutations in $\mathcal{D}_7$ that begin with letter $3$:
$$3421657,\ \ 3564217,\ \ 3642157.$$
Kreweras~\cite{Kre} proved that the triangle $K_{2m-1,k}$ ($1\leq k\leq 2m-1$)
shares the same recurrence relation as $\tilde{L}_{2m-1,k}$ in Lemma~\ref{rec:krew}. This leads to the following refinement of Conjecture~\ref{conj:Sun}.

\begin{theorem}\label{per:krew} Let $m$ be a positive integer. For each
$k=1,2,\ldots, 2m-1$, we have $\tilde{L}_{2m-1,k}=K_{2m-1,k}$. In particular, $\per(\tilde{L}_{2m-1})=(-1)^{m+1}G_{2m+2}$, and hence Conjecture~\ref{conj:Sun} holds.
\end{theorem}

\subsection{Four permutation interpretations of Kreweras' triangle}

  It is interesting to point out that Theorem~\ref{per:krew} is equivalent to the following new permutation interpretation of Kreweras' triangle.

\begin{corollary}
The Kreweras number  $K_{2m-1,k}$ with $1\le k\le 2m-1$ enumerates permutations $\sigma\in\S_{2m-1}$ satisfying $\sigma^{-1}(2m-1)=k$ and $\sigma(i)\notin[\lfloor\frac{i}{2}\rfloor+1, \lceil\frac{i}{2}\rceil+m-2]$ for each $i\in[2m-1]$.
\end{corollary}

Let $\gamma_{n,i}$ be the row vector with the first $i$ entries being $0$ and the remaining $n-i$ entries being $1$. Let  $\bar\gamma_{n,i}:={\bf1}^T_n-\gamma_{n,i}$ be the complement of $\gamma_{n,i}$, where  ${\bf1}_n$ is the $n
$-dimensional column vector with all entries equal to $1$.  For brevity,  we write $\gamma_{n,i}$ as $\gamma_i$ when $n$ is fixed. Observe that
$$
\tilde{L}_{2m-1}=(\gamma_1,\gamma_3,\ldots,\gamma_{2m-3},\bar\gamma_2,\bar\gamma_4,\ldots,\bar\gamma_{2m-2},{\bf1}^T_{2m-1})^T.
$$
Rearrange the rows of $\tilde{L}_{2m-1}$ to form another matrix
$$
L^{\ast}_{2m-1}=(\bar\gamma_2,\bar\gamma_4,\ldots,\bar\gamma_{2m-2},{\bf1}^T_{2m-1},\gamma_1,\gamma_3,\ldots,\gamma_{2m-3})^T
$$ with the same permanent. For example,
\begin{align*}
L^{\ast}_3=\begin{bmatrix}
1 & 1 & 0 \\
1 & 1 & 1\\
0 & 1 & 1
\end{bmatrix},
\;
L^{\ast}_5=\begin{bmatrix}
1 & 1 & 0 & 0 & 0\\
1 & 1 & 1 & 1 & 0\\
1 & 1 & 1 & 1 & 1\\
0 & 1 & 1 & 1 & 1\\
0 & 0 & 0 & 1 & 1
\end{bmatrix},
\;
L^{\ast}_7=\begin{bmatrix}
1 & 1 & 0 & 0 & 0 & 0 & 0\\
1 & 1 & 1 & 1 & 0 & 0 & 0\\
1 & 1 & 1 & 1 & 1 & 1 & 0\\
1 & 1 & 1 & 1 & 1 & 1 & 1\\
0 & 1 & 1 & 1 & 1 & 1 & 1\\
0 & 0 & 0 & 1 & 1 & 1 & 1\\
0 & 0 & 0 & 0 & 0 & 1 & 1
\end{bmatrix}.
\end{align*}
 Thus, Theorem ~\ref{per:krew} implies the following alternative permutation  interpretation of Kreweras' triangle.

\begin{corollary}
The Kreweras number  $K_{2m-1,k}$ enumerates permutations $\sigma\in\S_{2m-1}$ satisfying $\sigma^{-1}(m)=k$ and $\lceil\frac{i}{2}\rceil\leq\sigma(i)\leq m+\lfloor\frac{i}{2}\rfloor$ for each $i\in[2m-1]$.
\end{corollary}

We may also rearrange the rows of $\tilde{L}_{2m-1}$ as
$$
L^{\star}_{2m-1}=(\gamma_1,\bar\gamma_2,\gamma_3,\bar\gamma_4,\ldots,\gamma_{2m-3},\bar\gamma_{2m-2},{\bf1}^T_{2m-1})^T.
$$
For example,
\begin{align*}
L^{\star}_3=\begin{bmatrix}
0 & 1 & 1\\
1 & 1 & 0 \\
1 & 1 & 1
\end{bmatrix},
\;
L^{\star}_5=\begin{bmatrix}
0 & 1 & 1 & 1 & 1\\
1 & 1 & 0 & 0 & 0\\
0 & 0 & 0 & 1 & 1\\
1 & 1 & 1 & 1 & 0\\
1 & 1 & 1 & 1 & 1
\end{bmatrix},
\;
L^{\star}_7=\begin{bmatrix}
0 & 1 & 1 & 1 & 1 & 1 & 1\\
1 & 1 & 0 & 0 & 0 & 0 & 0\\
0 & 0 & 0 & 1 & 1 & 1 & 1\\
1 & 1 & 1 & 1 & 0 & 0 & 0\\
0 & 0 & 0 & 0 & 0 & 1 & 1\\
1 & 1 & 1 & 1 & 1 & 1 & 0\\
1 & 1 & 1 & 1 & 1 & 1 & 1
\end{bmatrix}.
\end{align*}
This rearrangement, together with Theorem~\ref{per:krew}, yields the following third interpretation of Kreweras' triangle.

\begin{corollary}\label{int3:krew}
 The Kreweras number  $K_{2m-1,k}$ enumerates permutations $\sigma\in\S_{2m-1}$
 for which $\sigma(2m-1)=k$, and $\sigma(i)>i$ if and only if $i\in\{2n-1:\ 1\ls n<m\}$.
\end{corollary}

A permutation $\sigma\in\S_{2m}$ is called a {\em Dumont permutation of the second kind}~\cite{Bur} if $\sigma(2i)<2i$ and $\sigma(2i-1)\geq 2i-1$ for every $i\in[m]$. It is clear that the permanent of the $2m\times2m$ matrix
$$
L^{\star}_{2m}=({\bf1}^T_{2m},\bar\gamma_1,\gamma_2,\bar\gamma_3,\gamma_4,\ldots,\bar\gamma_{2m-3},\gamma_{2m-2},\bar\gamma_{2m-1})^T
$$ enumerates Dumont permutations of the second kind in $\S_{2m}$.
For example,
\begin{align*}
L^{\star}_4=\begin{bmatrix}
1&1&1&1\\
1 & 0& 0&0\\
0 & 0&1&1 \\
1 & 1 & 1&0
\end{bmatrix},
\;
L^{\star}_6=\begin{bmatrix}
1 & 1 & 1 & 1 & 1&1\\
1&0 & 0 & 0 & 0 & 0\\
0&0 & 1 & 1 & 1 & 1\\
1&1 & 1 & 0 & 0 & 0\\
0&0 & 0 & 0 & 1 & 1\\
1&1 & 1 & 1 & 1 & 0
\end{bmatrix},
\;
L^{\star}_8=\begin{bmatrix}
1&1 & 1 & 1 & 1 & 1 & 1 & 1\\
1&0& 0 & 0 & 0 & 0 & 0 & 0\\
0&0 & 1 & 1 & 1 & 1 & 1 & 1\\
1&1 & 1 & 0 & 0 & 0 & 0 & 0\\
0&0 & 0 & 0 & 1 & 1 & 1 & 1\\
1&1 & 1 & 1 & 1 & 0 & 0 & 0\\
0&0 & 0 & 0 & 0 & 0 & 1 & 1\\
1&1 & 1 & 1 & 1 & 1 & 1 & 0
\end{bmatrix}.
\end{align*}
Observe that $\bar\gamma_1=(1,0,0,\ldots,0)$.  Deleting the second row and the first column from $L^{\star}_{2m}$ yields the $(2m-1)\times(2m-1)$-matrix $({\bf1}^T_{2m-1},\gamma_1,\bar\gamma_2,\gamma_3,\bar\gamma_4,\ldots,\gamma_{2m-3},\bar\gamma_{2m-2})^T$, which is clearly a row rearrangement of $L^{\star}_{2m-1}$.
This, together with Corollary~\ref{int3:krew}, implies the following interpretation of Kreweras' triangle in terms of Dumont permutations of the second kind.

\begin{corollary}\label{int4:krew}
Let $\mathcal{D}^2_{2m}$ be the set of Dumont  permutations of the second kind with length $2m$. Then
 the Kreweras number  $K_{2m-1,k}$ enumerates $\sigma\in\mathcal{D}^2_{2m}$ satisfying $\sigma(1)=k+1$.
\end{corollary}

\begin{remark}
A variation of Foata's first fundamental transformation\footnote{The original form of Foata's first fundamental transformation will be used in Section~\ref{tran:FS}.} (cf.~\cite[Prop. 1.3.1]{St}) establishes a one-to-one correspondence between $\{\sigma\in\mathcal{D}^2_{2m}: \sigma(1)=k+1\}$ and $\{\sigma\in\mathcal{D}_{2m+1}: \sigma(1)=k+1\}$, which provides an alternative approach to Corollary~\ref{int4:krew}. In fact, we first write $\sigma\in\mathcal{D}^2_{2m}$ as a product of disjoint cycles meeting the following requirements:
 \begin{itemize}
 \item[(a)] each cycle is written with the last element being the smallest one in the cycle;
 \item[(b)] the cycles are arranged in increasing order of their smallest elements.
 \end{itemize}
 Now, define $\hat\sigma$ as the word (or permutation) obtained  by first erasing the parentheses in the above form of $\sigma$ and then adding the letter $2m+1$ to the end. Then $\sigma\mapsto\hat\sigma$ is the desired bijection between $\mathcal{D}^2_{2m}$ and $\mathcal{D}_{2m+1}$ such that $\sigma(1)=\hat\sigma(1)$.
\end{remark}

\subsection{Proof of Conjecture~\ref{conj:Lus}}
 We first make some observations similar to Fact~\ref{fact} by determining the sign, rearranging the columns$\slash$rows, and deleting the two rows which contain a unique $1$ or $-1$, etc.. This leads us to consider the matrix $\tilde{M}_{2n}=[m_{i,j}]_{1\le i,j\le 2n}$ (induced from $M_{2n+2}$), where
\begin{align*}
m_{i,j}=\begin{cases}
1 & \text{if $\lceil\frac{i}{2}\rceil\le j\le n+\lceil\frac{i}{2}\rceil$,}\\
0 & \text{otherwise.}
\end{cases}
\end{align*}
The first three matrices of this type are given by
\begin{align*}
\tilde{M}_2=\begin{bmatrix}
1 & 1 \\
1 & 1
\end{bmatrix},
\;
\tilde{M}_4=\begin{bmatrix}
1 & 1 & 1 & 0\\
1 & 1 & 1 & 0\\
0 & 1 & 1 & 1\\
0 & 1 & 1 & 1
\end{bmatrix},
\;
\tilde{M}_6=\begin{bmatrix}
1 & 1 & 1 & 1 & 0 & 0\\
1 & 1 & 1 & 1 & 0 & 0\\
0 & 1 & 1 & 1 & 1 & 0\\
0 & 1 & 1 & 1 & 1 & 0\\
0 & 0 & 1 & 1 & 1 & 1\\
0 & 0 & 1 & 1 & 1 & 1
\end{bmatrix}.
\end{align*}
Conjecture~\ref{conj:Lus} has the following equivalent version:
\begin{align}
\per(\tilde{M}_{2n})=H_{2n+1}\ \ \text{ for all } n\ge 1.
\end{align}

A salient feature of the matrix $\tilde{M}_{2n}$ is that the $(2i-1)$-th and $2i$-th rows are identical, for each $1\le i\le n$. Therefore, when we expand the matrix to compute its permanent, our choices for the $(2i-1)$-th and $2i$-th rows can always be swapped. This is in agreement with the fact that $H_{2n+1}$ is divisible by $2^n$. Actually, the integers $h_n:=H_{2n+1}/2^n\ (n\ge1)$ are usually referred to as the {\it normalized median Genocchi numbers}.

The permanent $\per(\tilde{M}_{2n})$ can be interpreted as a sum over certain subset of permutations in $\S_{2n}$. Recall that $$\DES(\pi):=\{i\in[n-1]:\pi(i)>\pi(i+1)\}$$ is the {\it descent set} of the permutation $\pi\in\S_n$. Let $\tilde{\S}_{2n}$ be the set of permutations $\pi=\pi(1)\cdots\pi(2n)$ in $\S_{2n}$ for which $\{1,3,5,\ldots,2n-1\}\subseteq \DES(\pi)$ and
 $\lceil\frac{i}{2}\rceil\le \pi(i)\le n+\lceil\frac{i}{2}\rceil$.
By the previous discussions, $2^n\cdot|\tilde{\S}_{2n}|=\per(\tilde{M}_{2n})$. Hence it remains to show that
\begin{align}\label{id:Dellac}
|\tilde{\S}_{2n}|=h_n.
\end{align}

Indeed, there is a well-known combinatorial model in the literature called the {\it Dellac configuration} \cite{HZ,Big}, which is known \cite{Fei} to be enumerated by the normalized median Genocchi numbers $h_n$.

\begin{Def}
A Dellac configuration of size $n$ is a tableau of width $n$ and height $2n$ that contains $2n$ dots between the line $y=x$ and $y=n+x$, such that each row contains exactly one dot and each column contains exactly two dots.
\end{Def}

\begin{proof}[{\bf Proof of Conjecture~\ref{conj:Lus}}]
Let $\mathrm{DC}_n$ be the set of Dellac configurations of size $n$. Now given a configuration $C\in \mathrm{DC}_n$, we read the $y$-coordinates of the dots in $C$, from the leftmost column to the rightmost and top-down inside each column. It is easy to check that this list of coordinates forms precisely a permutation in $\tilde{\S}_{2n}$. This gives a bijection between $\mathrm{DC}_n$ and $\tilde{\S}_{2n}$, and thus \eqref{id:Dellac} holds true.
\end{proof}

\section{Proofs of Conjectures~\ref{conj:ChenA}\texorpdfstring{$\sim$}{}\ref{conj:Chen2}}\label{sec:conj of chen}

A permutation $\pi\in\S_n$ is called a {\it derangement} if $\pi(i)\not=i$ for all $i\in[n]$.
For convenience, we set
$$\D_n :=\{\pi\in\S_n:\pi(i)\neq i\text{ for all }i\in[n]\}. $$
Let us recall the following two classical results:
\begin{align}
\label{Euler-sign}
\sum_{\pi\in\S_n}(-1)^{\exc(\pi)}&=\begin{cases}
(-1)^mE_{2m+1} & \text{if $n=2m+1$},\\
0 & \text{if $n=2m$},
\end{cases}\\
\label{Roselle-sign}
\sum_{\pi\in\D_n}(-1)^{\exc(\pi)}&=\begin{cases}
0 & \text{if $n=2m+1$},\\
(-1)^mE_{2m} & \text{if $n=2m$},
\end{cases}
\end{align}
where $\exc(\pi)$ is the number of {\it excedances} of $\pi$, i.e.,
$$
\exc(\pi) :=|\{i\in [n]:\pi(i)>i\}|.
$$

Note that~\eqref{Euler-sign} and~\eqref{Roselle-sign} are due to Euler~\cite{Eul} and Roselle~\cite{Ros} respectively, and a joint combinatorial proof of them can be found in \cite[Chap.~5]{FS} (see also~\cite{LZ}). We emphasize that the proofs of Conjectures \ref{conj:ChenA}\texorpdfstring{$\sim$}{}\ref{conj:Chen2} hinge on linking these conjectures to the above two identities. In particular, the proof of Conjecture~\ref{conj:ChenA} is more involved and it requires block matrix decomposition and a variant of excedance that we denote as ``$\exph$''. This proof is thus given after the proofs of Conjectures~\ref{conj:Chen} and \ref{conj:Chen2}.

\subsection{Proof of Conjecture~\ref{conj:Chen}}

In view of the three matrices
\begin{align*}
P_2=\begin{bmatrix}
1 & 0 \\
0 & -1
\end{bmatrix},
\;
P_4=\begin{bmatrix}
1 & 1 & 1 & 0\\
1 & 1 & 0 & -1\\
1 & 0 & -1 & -1\\
0 & -1 & -1 & -1
\end{bmatrix},
\;
P_6=\begin{bmatrix}
1 & 1 & 1 & 1 & 1 & 0\\
1 & 1 & 1 & 1 & 0 & -1\\
1 & 1 & 1 & 0 & -1 & -1\\
1 & 1 & 0 & -1 & -1 & -1\\
1 & 0 & -1 & -1 & -1 & -1\\
0 & -1 & -1 & -1 & -1 & -1
\end{bmatrix},
\end{align*}
it is not hard to see the connection between the pattern of signs of $P_{2n}$ and the permutation statistic $\exc(\pi)$. More precisely, if we interpret each $-1$ entry of $P_{2n}$ sitting at the $i$-th row (counting from top to bottom) and the $j$-th column (counting from right to left) as an excedance $i=\pi(j)>j$, then
\begin{equation}\label{eq:signder}
\per(P_{2n})=\sum_{\pi\in\D_{2n}}(-1)^{\exc(\pi)}.
\end{equation}
Applying \eqref{Roselle-sign} we get $\per(P_{2n})=(-1)^nE_{2n}$.

A straightforward computation checks that the inverse matrix of $P_{2n}$ is $[x_{i,j}]_{1\le i,j \le 2n}$, where
$$x_{i,j}=\begin{cases}
0 & \text{if $i+j=2n+1$},\\
1 & \text{if $i+j<2n+1$ and $i\equiv j \pmod 2$,}\\
& \text{or $i+j>2n+1$ and $i\not\equiv j \pmod 2$},\\
-1 & \text{otherwise}.
\end{cases}$$
 The first three inverse matrices are
\begin{align*}
P^{-1}_2=\begin{bmatrix}
1 & 0 \\
0 & -1
\end{bmatrix},
\;
P^{-1}_4=\begin{bmatrix}
1 & -1 & 1 & 0\\
-1 & 1 & 0 & -1\\
1 & 0 & -1 & 1\\
0 & -1 & 1 & -1
\end{bmatrix},
\;
P^{-1}_6=\begin{bmatrix}
1 & -1 & 1 & -1 & 1 & 0\\
-1 & 1 & -1 & 1 & 0 & -1\\
1 & -1 & 1 & 0 & -1 & 1\\
-1 & 1 & 0 & -1 & 1 & -1\\
1 & 0 & -1 & 1 & -1 & 1\\
0 & -1 & 1 & -1 & 1 & -1
\end{bmatrix}.
\end{align*}

For each permutation $\pi\in\S_n$, an index $i\in [n]$ is called an {\it excedance of type $P$}, if either $\pi(i)>i$ and $\pi(i)-i$ is odd, or $\pi(i)<i$ and $\pi(i)-i$ is even. Let $\exc_P(\pi)$ denote the number of type $P$ excedances of $\pi$. Note that $\exc_P$ is not an Eulerian statistic, i.e., its distribution over $\S_n$ is different from that of $\exc$. Nonetheless, $\exc_P$ and $\exc$ have the same sign-balance over both $\S_n$ and $\D_n$.

\begin{theorem}
For each $n\ge 1$, we have
\begin{equation}\label{id:exc=excP over der}
\sum_{\pi\in\S_n}(-1)^{\exc(\pi)} = \sum_{\pi\in\S_n}(-1)^{\exc_P(\pi)}
\ \t{and}\ \sum_{\pi\in\D_n}(-1)^{\exc(\pi)} = \sum_{\pi\in\D_n}(-1)^{\exc_P(\pi)}.
\end{equation}
\end{theorem}

\begin{proof}
For our convenience, we shall consider inverse excedence ($\iexc(\pi):=\exc(\pi^{-1})$) instead of excedance. Note that $\exc$ and $\iexc$ are equidistributed over both $\S_n$ and $\D_n$.

The idea is to begin with the cycle decomposition of certain permutation $\pi\in\S_{n-1}$, and insert $n$ into $\pi$, say between $i$ and $\pi(i)$, to get a new permutation $\pi'\in\S_n$.

\begin{table}[htb]
{
\begin{tabular*}{3.8in}{cccccc}
\toprule
$i$ & $\pi(i)$ & $\pi(i)-i$ & $n$ & $\iexc$ & $\exc_P$\\
\midrule
even & even & $>0$ & even/odd & $+1$ & $+1$\\
even & even & $<0$ & even/odd & $+0$ & $+0$\\
even & odd & $>0$ & even & $+1$ & $-1$\\
even & odd & $>0$ & odd & $+1$ & $+1$\\
even & odd & $<0$ & even & $+0$ & $+0$\\
even & odd & $<0$ & odd & $+0$ & $+2$\\
odd & even & $>0$ & even & $+1$ & $+1$\\
odd & even & $>0$ & odd & $+1$ & $-1$\\
odd & even & $<0$ & even & $+0$ & $+2$\\
odd & even & $<0$ & odd & $+0$ & $+0$\\
odd & odd & $>0$ & even/odd & $+1$ & $+1$\\
odd & odd & $<0$ & even/odd & $+0$ & $+0$\\
even/odd & even/odd & $=0$ & even/odd & $+1$ & $+1$\\
\bottomrule\\
\end{tabular*}
}
\caption{Various cases of inserting $n$ between $i$ and $\pi(i)$}
\end{table}

The rest of the proof is to verify, in a case-by-case fashion, that the parity changes of both $\iexc$ and $\exc_P$ are the same, when we go from $\pi$ to $\pi'$. When $n$ is a fixed point for $\pi'$, we see neither $\iexc$ nor $\exc_P$ changes in this case. We collect the remaining cases in the above table so that the readers can easily check for themselves. The bottom case in the table means that $n$ joins a fixed point $\pi(i)=i$ of $\pi$ to form a $2$-cycle $(i~n)$ of $\pi'$.
\end{proof}

By the definitions of the inverse matrix $P^{-1}_{2n}$ and the new statistic $\exc_P$, we have
$$\per(P^{-1}_{2n})=\sum_{\pi\in\D_{2n}}(-1)^{\exc_P(\pi)}.$$
This result combined with \eqref{Roselle-sign} and the second identity in \eqref{id:exc=excP over der} proves the equality $\per(P^{-1}_{2n})=(-1)^nE_{2n}$.

There is another elementary proof of the identity
\begin{equation}\label{Hadamard}
\per(P_{2n})=\per(P^{-1}_{2n})
\end{equation}
that we are going to provide below. The following elementary transformations concerning evaluations  of permanents are clear.

\begin{lemma}\label{ele:per}
Let $A=[a_{i,j}]_{1\leq i,j\leq n}$ be an $n\times n$ matrix over a commutative ring $R$ with identity.
\begin{enumerate}[{\rm (1)}]
\item For any $1\leq k\leq n$,  the matrix obtained from $A$ by multiplying each entry in the $k$-th column (or row) by $c\in R$ has permanent that equals $c\cdot\per(A)$.
\item Exchanging any two columns (or rows) of $A$ preserves the permanent.
\end{enumerate}
\end{lemma}

\begin{Def}[Hadamard product]  For two matrices $A=[a_{i,j}]_{1\leq i,j\leq n}$ and $B=[b_{i,j}]_{1\leq i,j\leq n}$ over a ring, the Hadamard product of $A$ and $B$ is $A\circ B=[c_{i,j}]_{1\leq i,j\leq n}$ with $c_{i,j}=a_{i,j}b_{i,j}$.
\end{Def}

\begin{Def}[An action on matrix]\label{Def:action}
Let $A=[a_{i,j}]_{1\leq i,j\leq n}$ be an $n\times n$ matrix over  $\mathbb{R}$. For $1\leq k,l\leq n$, define  $\phi_{k,l}(A)$ to be the matrix obtained from $A$ by multiplying the $k$-th row and the $l$-th column by $-1$. Note that the $(k,l)$-entry of $\phi_{k,l}(A)$ remains $a_{k,l}$, since it multiplies  $-1$ twice. By Lemma~\ref{ele:per}~(1), we have
\begin{equation}\label{action}
\per(\phi_{k,l}(A))=\per(A).
\end{equation}
\end{Def}
The usefulness of $\phi_{k,l}$ is demonstrated by the next lemma, and we shall derive further properties of this action in the last section; see Proposition~\ref{char-action}.
\begin{lemma}\label{H:sign}
Let $A=[a_{i,j}]_{1\leq i,j\leq n}$ be an $n\times n$ matrix over  $\mathbb{R}$. Then
$$
\biggl(\prod_{k+l\leq n}\phi_{k,l}\biggr)(A)=A\circ H_n,
$$
where $H_n=[h_{i,j}]_{1\leq i,j\leq n}$ with $h_{i,j}=(-1)^{i+j}$. Consequently, $\per(A)=\per(A\circ H_n)$.
\end{lemma}
\begin{proof}
The number of times that the entry $a_{i,j}$ changes its sign under $\prod_{k+l\leq n}\phi_{k,l}$ equals $n-i+n-j=2n-(i+j)$, which has the same parity as $i+j$. This proves the first statement in Lemma \ref{H:sign}. The second statement follows from the first one and the equality \eqref{action}.
\end{proof}

In view of the relation
$$
P^{-1}_{2n}=P_{2n}\circ H_{2n},
$$
 Lemma~\ref{H:sign} provides an alternative approach to~\eqref{Hadamard}.

 \subsection{Proof of Conjecture~\ref{conj:Chen2}}
 The first few terms of $Q_n$ are
 \begin{align*}
 &Q_1=\begin{bmatrix}
 -1
\end{bmatrix},
\;
Q_2=\begin{bmatrix}
0 & -1 \\
-1 & 0
\end{bmatrix},
\;
Q_3=\begin{bmatrix}
1 & -1 & -1 \\
0& -1  & 0\\
-1& -1 &1
\end{bmatrix},
\;
Q_4=\begin{bmatrix}
1 & 0 & -1 & -1\\
1 & -1 & -1 & 0\\
0 & -1 & -1 & 1\\
-1 & -1 & 0 & 1
\end{bmatrix},
\;\\
& Q_5=\begin{bmatrix}
1 & 1 & -1 & -1&-1\\
1 & 0 &-1 & -1 & 0\\
1 & -1 & -1 &-1 & 1\\
0& -1 & -1 & 0 & 1\\
-1&-1&-1&1&1
\end{bmatrix},
\quad
Q_6=\begin{bmatrix}
1 & 1 & 0 & -1 & -1 & -1\\
1 & 1 & -1 & -1 & -1 & 0 \\
1 & 0 & -1 & -1 & -1 & 1\\
1 & -1 & -1 & -1 & 0 & 1\\
0 & -1& -1 & -1 & 1 & 1\\
 -1 & -1 & -1 & 0 & 1 & 1
\end{bmatrix},\\
&Q_7=\begin{bmatrix}
1 & 1 & 1 & -1 & -1 & -1& -1\\
1 & 1 & 0 & -1 & -1 & -1 & 0 \\
1 & 1 & -1 & -1 & -1 &-1 & 1\\
1 &0& -1 & -1 & -1 & 0 & 1\\
1 & -1 & -1 & -1 & -1 &1 & 1\\
0 & -1& -1 & -1 &0& 1 & 1\\
 -1 & -1 & -1 & -1 & 1 & 1 & 1
\end{bmatrix},
\quad
Q_8=\begin{bmatrix}
1 & 1 & 1 &0& -1 & -1 & -1& -1\\
1 & 1 & 1 &-1 & -1 & -1 & -1 & 0 \\
1 & 1 & 0& -1 & -1 & -1 &-1 & 1\\
1 &1& -1 & -1 & -1 &-1 & 0 & 1\\
1 & 0&  -1 & -1 & -1 & -1 &1 & 1\\
1 &-1& -1 & -1 & -1  & 0 & 1&1\\
0 & -1& -1 & -1 & -1& 1& 1 & 1\\
 -1 & -1 & -1 & -1 & 0 & 1 & 1 & 1
\end{bmatrix}.
\end{align*}

For $0\leq i\leq n-1$, let $\alpha_{n,i}$ denote the column vector of dimension $n$ whose $(n-i)$-th entry is $0$ and all entries above (resp.~below) this zero entry are $1$ (resp.~$-1$). For example, $\alpha_{6,2}=(1,1,1,0,-1,-1)^T$. For the sake of simplicity,  we write $\alpha_{n,i}$ as $\alpha_i$ when $n$ is fixed.
The matrix $Q_n$ has the following structure.

\begin{lemma}\label{mat:Q}
We have
$$Q_{2n}=(\alpha_{1},\alpha_{3},\ldots,\alpha_{2n-1},-\alpha_0,-\alpha_2,\ldots,-\alpha_{2n-2})$$
and $$Q_{2n+1}=(\alpha_{1},\alpha_{3},\ldots,\alpha_{2n-1},-{\bf1}_{2n+1},-\alpha_1,-\alpha_3,\ldots,-\alpha_{2n-1}).$$
\end{lemma}

The following lemma is clear from the definition of permanents.
\begin{lemma}\label{obs:rep}
For a matrix with two identical columns having one zero at the same position, replacing the zero by $1$ and the other zeros by $-1$ does not change its permanent.
\end{lemma}

\begin{proof}[{\bf Proof of Conjecture~\ref{conj:Chen2}}]
In view of \eqref{P2n},
\begin{equation}\label{eq:signder1}
\per(P_{2n})=(-1)^nE_{2n}.
\end{equation}
By Lemmas~\ref{ele:per} and~\ref{mat:Q}, we have
\begin{align*}
\per(Q_{2n})&=(-1)^n\per(\alpha_{1},\alpha_{3},\ldots,\alpha_{2n-1},\alpha_0,\alpha_2,\ldots,\alpha_{2n-2})\\
&=(-1)^n\per(\alpha_{0},\alpha_{1},\ldots,\alpha_{2n-1})=(-1)^n\per(P_{2n}).
\end{align*}
Combining this with \eqref{eq:signder1}, we see that Conjecture~\ref{conj:Chen2} holds for even $n$.

It remains to deal with Conjecture~\ref{conj:Chen2} for odd $n$. Let $\beta_i$ be the column vector obtained by replacing the only zero in $\alpha_i$ by $1$. By Lemmas~\ref{ele:per} and~\ref{obs:rep}, we have
\begin{align*}
\per(Q_{2n+1})&=(-1)^{n}\per(\alpha_{1},\alpha_{3},\ldots,\alpha_{2n-1},-{\bf1}_{2n+1},\alpha_1,\alpha_3,\ldots,\alpha_{2n-1})\\
&=(-1)^{n}\per(\alpha_{1},\alpha_{1},\alpha_{3},\alpha_{3},\ldots,\alpha_{2n-1},\alpha_{2n-1},-{\bf1}_{2n+1})\\
&=(-1)^n\per(\beta_1,\beta_2,\ldots,\beta_{2n-1},\beta_{2n},-{\bf1}_{2n+1})\\
&=(-1)^n\sum_{\sigma\in\S_{2n+1}}(-1)^{\wexc(\sigma)},
\end{align*}
where $\wexc(\sigma):=|\{i\in[2n+1]: \sigma(i)\geq i\}|$ is the number of weak excedances of $\sigma$.
In view of the permutation interpretation~\eqref{Euler-sign} of tangent numbers, we have
\begin{align*}
\sum_{\sigma\in\S_{2n+1}}(-1)^{\wexc(\sigma)}&=\sum_{\sigma\in\S_{2n+1}}(-1)^{2n+1-|\{i\in[2n+1]:\, \sigma(i)< i\}|}\\
&=-\sum_{\sigma\in\S_{2n+1}}(-1)^{|\{i\in[2n+1]:\, \sigma(i)< i\}|}\\
&=-\sum_{\sigma\in\S_{2n+1}}(-1)^{\exc(\sigma^{-1})}\\
&=-\sum_{\sigma\in\S_{2n+1}}(-1)^{\exc(\sigma)}=(-1)^{n+1}E_{2n+1},
\end{align*}
which proves the odd case of Conjecture~\ref{conj:Chen2}.
\end{proof}

\subsection{Proof of Conjecture~\ref{conj:ChenA}}
We will make use of the multiplication of block matrices to facilitate our computation. For every $n\ge 1$, we introduce two $n\times n$ matrices:
$$I_n=\begin{bmatrix}
1 & 0 & \cdots & 0\\
0 & \ddots & \ddots & \vdots\\
\vdots & \ddots & \ddots & 0\\
0 & \cdots & 0 & 1
\end{bmatrix}\ \ \text{and}\ \
J_n=\begin{bmatrix}
0 & \cdots & 0 & 1\\
\vdots & \iddots & \iddots & 0\\
0 & \iddots & \iddots & \vdots\\
1 & 0 & \cdots & 0
\end{bmatrix}.$$
Recall the matrix $P_n=[\sgn(\sin(\frac{i+j}{n+1}\pi))]_{1\le i,j\le n}$ defined in Conjecture~\ref{conj:Chen}. It is easy to verify that
\begin{align}
\label{id:IJP}
J_n^2=I_n\quad \t{and}\quad J_nP_n=-P_nJ_n.
\end{align}

We derive the following block decomposition of the matrix $A_{2n}$ from its definition:
\begin{align*}
A_{2n}=\begin{bmatrix}
P_n-J_n & -P_n \\
-P_n & P_n+J_n
\end{bmatrix}.
\end{align*}
 With the aid of \eqref{id:IJP}, it is fairly easy to check that the inverse matrix of $A_{2n}$ has the following block decomposition:
\begin{align*}
A^{-1}_{2n}=\begin{bmatrix}
P_n-J_n & P_n \\
P_n & P_n+J_n
\end{bmatrix}.
\end{align*}
In other words, $A^{-1}_{2n}=[\tilde{a}_{i,j}]_{1\le i,j\le 2n}$, where
\begin{align*}
\tilde{a}_{i,j}:=\begin{cases}
0 & \text{if $i+j=2n+1$},\\
-1 & \text{if $i+j\ge n+1$ and $\max(i,j)\le n$},\\
 & \text{or if $i+j\ge 2n+2$ and $i\le n$},\\
 & \text{or if $i+j\ge 2n+2$ and $j\le n$},\\
 & \text{or if $i+j\ge 3n+2$},\\
1 & \text{otherwise}.
\end{cases}
\end{align*}
Comprehending the sign pattern of $A^{-1}_{2n}$, we are naturally led to define the following variant of excedance, that we call the number of {\it excedances induced by $\phi$} (a bijection to be defined later), and denoted as $\exph(\pi)$ for each permutation $\pi\in\S_{2n}$. Each $i\in[2n]$ is said to be an excedance of $\pi$ induced by $\phi$, if it falls in one of the following situations.
\begin{enumerate}[I.]
  \item $1\le i<\pi(i)\le n$;
  \item $n+1\le i\le 2n$ and $i-n\le \pi(i)\le n$;
  \item $n+1\le i<\pi(i)\le 2n$;
  \item $1\le i<\pi(i)-n$ and $n+1\le \pi(i)\le 2n$.
\end{enumerate}

Now we can interprete $\per(A_{2n}^{-1})$ as the signed sum over all derangements of $[2n]$. Namely, each $\pi\in\D_{2n}$ corresponds to one term in the expansion of $\per(A_{2n}^{-1})$, and each pair $(i,\pi(i))$ corresponds to the entry in the $\pi(i)$-th (from top to bottom) row and the $i$-th (from right to left) column of $A_{2n}^{-1}$. This entry is $-1$ if and only if $i$ is an excedance of $\pi$ introduced by $\phi$; otherwise it is $1$. In other words, we have
\begin{align}\label{id:per=exph over der}
\per(A_{2n}^{-1})=\sum_{\pi\in\D_{2n}}(-1)^{\exph(\pi)}.
\end{align}

We give an example to illustrate the above relation.
\begin{example}
The six entries in $A_6^{-1}$ that correspond to the derangement $\pi=315624$ have been colored blue. Note that $\exph(\pi)=3$ since $\pi$ has excedances introduced by $\phi$ at $1$ (case I), $4$ (case III), and $5$ (case II).
\begin{align*}
A_6^{-1}=\begin{bmatrix}
1 & 1 & -1 & 1 & \blue{1} & 0\\
1 & \blue{-1} & -1 & 1 & 0 & -1\\
-1 & -1 & -1 & 0 & -1 & \blue{-1}\\
\blue{1} & 1 & 0 & 1 & 1 & 1\\
1 & 0 & -1 & \blue{1} & 1 & -1\\
0 & -1 & \blue{-1} & 1 & -1 & -1
\end{bmatrix}
\quad \Longrightarrow
\quad  \exph(\pi)=3\ \text{for}\ \pi=315624.
\end{align*}
\end{example}

Relying on the connection \eqref{id:per=exph over der}, we proceed to show that
\begin{align}
\per(A^{-1}_{2n})=(-1)^nE_{2n}.
\label{A^-1=secant}
\end{align}
Indeed, we will construct a bijection to show the following equidistribution result, which immediately implies \eqref{A^-1=secant} when we set $t=-1$ and $y=0$, and apply \eqref{Roselle-sign}. For any $\pi\in\S_n$, let $\Fix(\pi)=\{i\in[n]:\pi(i)=i\}$ be the set of fixed points of $\pi$ and $\fix(\pi)=|\Fix(\pi)|$ be its cardinality.
\begin{theorem}\label{thm:bij for A^-1}
For each $n\ge 1$, we have
\begin{align}
\label{id:exph and fix}
\sum_{\pi\in\S_{2n}}t^{\exc(\pi)}y^{\fix(\pi)}=\sum_{\pi\in\S_{2n}}t^{\exph(\pi)}y^{\fix(\pi)}.
\end{align}
\end{theorem}

 For each $x\in [2n]$, we define $y:=\phi(x)$ by
\begin{align}\label{def of phi}
y=\begin{cases}
n+k & \text{if $x=2k-1$ for some $1\le k\le n$},\\
k & \text{if $x=2k$ for some $1\le k\le n$}.
\end{cases}
\end{align}
Then $\phi$ is a permutation of $[2n]$, and it naturally induces a bijection on the symmetric group $\S_{2n}$:
\begin{align*}
\Phi: \S_{2n}&\to\S_{2n}\\
\pi &\mapsto \sigma,
\end{align*}
where $\sigma$ is obtained from $\pi$ by replacing each $i$ with $\phi(i)$ in the two-line notation of $\pi$. Take $\pi=315462$ for example, we see
$$\begin{pmatrix}
1 & 2 & 3 & 4 & 5 & 6 \\
3 & 1 & 5 & 4 & 6 & 2
\end{pmatrix}\mapsto
\begin{pmatrix}
4 & 1 & 5 & 2 & 6 & 3 \\
5 & 4 & 6 & 2 & 3 & 1
\end{pmatrix}=
\begin{pmatrix}
1 & 2 & 3 & 4 & 5 & 6 \\
4 & 2 & 1 & 5 & 6 & 3
\end{pmatrix},$$
so we have $\sigma=\Phi(\pi)=421563$. One checks that $\exc(\pi)=\exph(\sigma)=3$ and $\fix(\pi)=\fix(\sigma)=1$.

\begin{proof}[Proof of Theorem~\ref{thm:bij for A^-1}]
We show that the previous example is not a coincidence. Namely, if $\Phi$ maps $\pi$ to $\sigma$, then we have
\begin{equation}
\label{id:exc=exph}
\exc(\pi) =\exph(\sigma)
\end{equation}
and
\begin{equation}
\label{id:fix}
\fix(\pi) =\fix(\sigma),
\end{equation}
which readily imply \eqref{id:exph and fix}.
Since $\phi$ is a bijection, $i=\pi(i)$ if and only if $\phi(i)=\phi(\pi(i))$, and actually $\phi(\pi(i))=\sigma(\phi(i))$. Consequently, $i$ is a fixed point of $\pi$ if and only if $\phi(i)$ is a fixed point of $\sigma$. This proves \eqref{id:fix}.

The proof of \eqref{id:exc=exph} is a case-by-case verification using the definition of the statistic ``$\exph$''. We show one case here and leave the details of remaining cases to the interested reader. Suppose that $i<\pi(i)$ is an excedance of $\pi$ such that $i=2k-1$ is odd while $\pi(i)=2j$ is even. According to \eqref{def of phi}, we have $\phi(i)=n+k$ while $\phi(\pi(i))=j\ge k=\phi(i)-n$. So the pair $(\phi(i),\phi(\pi(i)))=(\phi(i),\sigma(\phi(i)))$ is in situation II and contributes $1$ to $\exph(\sigma)$.
\end{proof}

To complete the proof of Conjecture~\ref{conj:ChenA}, it suffices now to show that
\begin{equation}\label{per:invA}
\per(A_{2n})=\per(A^{-1}_{2n}).
\end{equation}
This is done by utilizing the  action $\phi_{k,l}$ on matrices  introduced in Definition~\ref{Def:action}.

\begin{lemma}\label{H:sign2}
Let $A=[a_{i,j}]_{1\leq i,j\leq 2n}$ be a $2n\times 2n$ matrix over  $\mathbb{R}$. Let $U_n$ be the $n\times n$ matrix with all entries $1$, and let
$$
\tilde U_{2n}=\begin{bmatrix}
U_n &-U_n\\
-U_n & U_n
\end{bmatrix}.
$$  Then
$$
\biggl(\prod_{(k,l)\in S}\phi_{k,l}\biggr)(A)=A\circ \tilde U_{2n},
$$
where $S=\{(i,j)\in[2n]\times[2n]: i+j\in[n+1,2n]\cup[3n+2,4n] \}$. Consequently, $\per(A)=\per(A\circ \tilde U_{2n})$.
\end{lemma}
\begin{proof}
The number of times that the entry $a_{i,j}$ changes its sign under $\prod\limits_{(k,l)\in S}\phi_{k,l}$ equals
$$
\begin{cases}
2n-1, &\quad\text{ if $(i,j)\in[1,n]\times[n+1,2n]\cup[n+1,2n]\times[1,n]$;}\\
2n-2, &\quad\text{ if $(i,j)\in[n+1,2n]\times[n+1,2n]$;}\\
2n, &\quad\text{ if $(i,j)\in[1,n]\times[1,n]$}.
\end{cases}
$$
This proves the first statement in Lemma \ref{H:sign2}.
The second statement follows from the first one and the equality~\eqref{action}.
\end{proof}

Since
$$
A^{-1}_{2n}=\begin{bmatrix}
P_n-J_n & P_n \\
P_n & P_n+J_n
\end{bmatrix}=A_{2n}\circ \tilde U_{2n},
$$
the desired identity~\eqref{per:invA} follows from Lemma~\ref{H:sign2}. This completes the proof of Conjecture~\ref{conj:ChenA}.

\section{Proof of Conjecture~\ref{conj:Chen3} and a variant of excedance}\label{sec:conj of chen2}

The proof of Conjecture~\ref{conj:Chen3} consists of three steps. Firstly, we adjust the original matrix $R_n$ to get $\tilde{R}_n$ so that the modified permanent $\per(\tilde{R}_n)$, for $n$ even (resp.~odd), can be recognized as the sign balance of certain variant of the excedance statistic over permutations (resp.~almost derangements). Next, with the aid of  the celebrated {\em Foata--Strehl action}~\cite{FSt} on permutations and a simple bijection, we are able to identify these weighted sums as the so-called {\it central Poupard numbers}. Finally, we confirm the conjectured connections with the binomial transforms of Euler numbers, namely Eq.~\eqref{bino:chen}, utilizing the generating functions of Poupard numbers due to Foata and Han~\cite{FH}. As a byproduct, we prove in subsection~\ref{subsec:Bala} a conjecture due to Peter Bala on a certain continued fraction expansion.

\subsection{An exc variant and cyclic valley-hopping}
\label{tran:FS}
The first few terms of $R_n\ (n\ge1)$ read as follows:
\begin{align*}
&R_1=\begin{bmatrix}
-1
\end{bmatrix},
\;
R_2=\begin{bmatrix}
-1 & -1 \\
-1 & -1
\end{bmatrix},
\;
R_3=\begin{bmatrix}
0 & -1 & -1 \\
-1 & -1  & -1\\
-1  & -1 & 0
\end{bmatrix},
\;
R_4=\begin{bmatrix}
1 & -1 & -1 & -1 &\\
-1 & -1 & -1 & -1 \\
-1 & -1 & -1 & -1 \\
-1 & -1 & -1 & 1
\end{bmatrix},\\
&\qquad R_5=\begin{bmatrix}
1 & 0&-1 & -1 & -1 &\\
0&-1 & -1 & -1 & -1 \\
-1&-1 & -1 & -1 & -1 \\
-1&-1 & -1 & -1 & 0 \\
-1 & -1 & -1 &0 & 1
\end{bmatrix},
\;
R_6=\begin{bmatrix}
1 & 1&-1&-1 & -1 & -1 \\
1 & -1&-1&-1 & -1 & -1 \\
-1 & -1&-1&-1 & -1 & -1 \\
-1 & -1&-1&-1 & -1 & -1 \\
-1 & -1&-1&-1 & -1 & 1 \\
-1 & -1&-1&-1 & 1 & 1
\end{bmatrix}.
\end{align*}
After multiplying the $i$-th row by $-1$ for each $\lceil\frac{n}{2}\rceil\le i\le n$, and rearranging the rows we get the  matrix $\tilde R_n$ as illustrated by the following examples:
\begin{align*}
&\tilde R_1=\begin{bmatrix}
\red{1}
\end{bmatrix},
\;
\tilde R_3=\begin{bmatrix}
1 & 1 & 0 \\
\red{1} & \red{1}  & \red{1}\\
0  & -1 & -1
\end{bmatrix},
\;
 \tilde R_5=\begin{bmatrix}
1 & 1 & 1 & 1 &0\\
 1 & 1 & 1 &0&-1\\
\red{1} & \red{1}  & \red{1}& \red{1}& \red{1} \\
1 & 0&-1&-1 & -1   \\
0 & -1&-1 & -1 & -1
\end{bmatrix},\\
&\tilde R_2=\begin{bmatrix}
1 & 1 \\
\red{1} & \red{1}
\end{bmatrix},
\;
\tilde R_4=\begin{bmatrix}
1 & 1 & 1 & 1 &\\
1 & 1 & 1 & -1 \\
\red{1} & \red{1}  & \red{1}& \red{1} \\
1 & -1 & -1 & -1
\end{bmatrix},
\;
\tilde R_6=\begin{bmatrix}
1 & 1&1&1 & 1 & 1 \\
1 & 1&1&1 & 1 & -1 \\
1 & 1&1&1 & -1 & -1 \\
\red{1} & \red{1}  & \red{1}& \red{1}& \red{1} & \red{1}\\
1 & 1&-1&-1 & -1 & -1 \\
1 & -1&-1&-1 & -1 & -1
\end{bmatrix}.
\end{align*}
Note that
\begin{align}
\label{tildeR}
\per(R_n)=(-1)^{\lceil\frac{n+1}{2}\rceil}\per(\tilde R_n).
\end{align}
The comparison between $\tilde R_n$ and $P_n$ suggests the following variation of excedance statistic.

\begin{Def}
For $\pi\in\S_n$, define
$$
\widetilde\exc(\pi):=|\{i\in[n]:\pi(i)>i \text{ and  $\pi(i)\neq \lceil(n+1)/2\rceil$}\}|.
$$
Note that $\exc(\pi)-\widetilde\exc(\pi)$ equals $1$ or $0$ depending on whether $\lceil\frac{n+1}{2}\rceil$ is an excedance top or not.
\end{Def}

The two identities below follow directly from the pattern of $\tilde R_n$ and the definition of $\widetilde\exc$.
\begin{align}
\label{perR-even}
\per(\tilde R_{2m}) &=\sum_{\pi\in\S_{2m}}(-1)^{\widetilde\exc(\pi)},\\
\per(\tilde R_{2m+1}) &=\sum_{\pi\in\tilde\D_{2m+1}}(-1)^{\widetilde\exc(\pi)},
\label{perR-odd}
\end{align}
where $\tilde\D_{2m+1}=\{\pi\in\S_{2m+1}:\Fix(\pi)\subseteq\{m+1\}\}$.
In order to understand the cancellation happened in computing the permanents $\per(\tilde R_n)$, we adopt Foata's first fundamental transformation and the Foata--Strehl action on permutations that we recall below.

In this section, we write each permutation in its {\it standard cycle form} according to the following convention:

(i) each cycle has its largest letter in the leftmost position;

(ii) the cycles are listed from left to right in increasing order of their largest letters.

\noindent For instance, the cycle form of $\pi=562437198$ is $\pi=(4)(715326)(98)$. Foata's ``transformation fondamentale'' $o:\S_n\to\S_n$ is defined as follows: For each $\pi\in\S_n$, the one-line notation of $o(\pi)$ is obtained from the standard cycle form of $\pi^{-1}$ by erasing all the parentheses. For example, $o(\pi)=471532698$ for $\pi=735412698$, since $\pi^{-1}=562437198$ has the cycle form $(4)(715326)(98)$.

For a permutation $\pi\in\S_n$, introduce the {\em set of excedance tops} of $\pi$ and the {\em set of descent tops} of $\pi$ by
$$
\Exct(\pi):=\{\pi(i): 1\leq i<n, i<\pi(i)\}\quad\text{and}\quad\Dest(\pi)=\{\pi(i): 1\leq i<n, \pi(i)>\pi(i+1)\},
$$
respectively. The following result is known.

\begin{lemma}\label{f:transfor}
For the bijection $o:\S_n\to\S_n$, we have $\Exct(\pi)=\Dest(o(\pi))$ for each $\pi\in\S_n$.
\end{lemma}

Fix a permutation $\pi\in\S_n$ and a letter $x\in[n]$. The {\it $x$-factorization} of $\pi$ is the unique decomposition $\pi=w_1w_2xw_3w_4$, where $w_2$ (resp.~$w_3$) is the maximal consecutive subword (possibly empty) immediately to the left (resp.~right) of $x$ whose letters are all smaller than $x$.   The Foata--Strehl action~\cite{FSt} $\varphi_x:\S_n\to\S_n$ can be defined by
\begin{equation}\label{def:FSact}
\varphi_x(\pi)=w_1w_3xw_2w_4.
\end{equation}

\begin{figure}[ht]
\setlength{\unitlength}{1mm}
\begin{picture}(50,38)\setlength{\unitlength}{1mm}
\thinlines

\put(-3,30){$+\infty$}
\put(0,30){\line(1,-2){6}}
\put(6,18){\circle*{1.3}}\put(3,14){$7$}
\put(6,18){\line(1,-2){6}}
\put(12,6){\circle*{1.3}}\put(10,2){$3$}

\put(12,6){\line(1,1){6}}
\put(18,12){\circle*{1.3}}\put(18,13){$5$}
\put(18,12){\line(1,-1){12}}
\put(21,9){\circle*{1.3}}\put(19.2,5.8){\red{$4$}}
\put(30,0){\circle*{1.3}}\put(28,-3){$1$}
\put(30,0){\line(1,1){24}}

\put(33,3){\circle*{1.3}}\put(34,1){$2$}
\put(45,15){\circle*{1.3}}\put(46,13){$6$}
\put(54,24){\circle*{1.3}}\put(54.2,25){$9$}

\put(54,24){\line(1,-1){3}}
\put(57,21){\circle*{1.3}}\put(56,17.5){$8$}
\put(57,21){\line(1,1){9}}
\put(63,30){$+\infty$}
\put(21,9){\dashline{1}(1,0)(18,0)}
\put(39,9){\vector(1,0){0.1}}
\end{picture}
\caption{The Foata--Strehl action $\varphi_x$ on $735412698$ with $x=4$.}
\label{FS:action}
\end{figure}

For  example, if $\pi=735412698$ and $x=4$, then the $x$-factorization yields $w_1=735$, $w_2=\emptyset$, $w_3=12$, and $w_4=698$. Thus, $\varphi_x(\pi)=735124698$; see Fig.~\ref{FS:action} for a visualization of the action $\varphi_x$.
Clearly $\varphi_x$ is an involution on $\S_n$ for every $x\in[n]$. Following Sun and Wang~\cite{SW}, we  define the action $\psi_x:\S_n\to\S_n$ by
$\psi_x(\pi):= o^{-1}(\varphi_x(o(\pi))).$
The action $\psi_x$ is  an involution on $\S_n$ (as $\varphi_x$ is an involution), and when $x$ is properly chosen, it endows us with a sign-reversing bijection that proves  Lemma~\ref{lem:cancel} in the following.

Given a permutation $\pi=\pi(1)\pi(2)\cdots\pi(n)$ in $\S_n$, we use the  conventions
$\pi(0)=\pi(n+1)=+\infty$ (resp.~$\pi(0)=0, \pi(n+1)=+\infty$) when $n$ is even (resp.~odd) and say
 that $\pi(i)\in[n]$ is:
\begin{itemize}
  \item a {\it double ascent value} of $\pi$ if $\pi(i-1)<\pi(i)<\pi(i+1)$;
  \item a {\it double descent value} of $\pi$ if $\pi(i-1)>\pi(i)>\pi(i+1)$.
\end{itemize}
Denote by $\Dasc(\pi)$ (resp. $\Ddes(\pi)$) the set of  double ascent (resp. descent) values of $\pi$.
When $n=2m$ is even, let $\S_{2m}^*$ be the set of permutations $\pi$ in $\S_{2m}$ such that $\Dasc(o(\pi))\cup\Ddes(o(\pi))=\{m+1\}$. When $n=2m+1$ is odd, let $\S_{2m+1}^*$ be the set of permutations $\pi$ in $\S_{2m+1}$ such that $\Dasc(o(\pi))\cup\Ddes(o(\pi))=\{m+1\}$. Note that if $x\in\Fix(\pi)$, then the convention $o(\pi)(0)=0$ guarantees that $x$ is a double ascent value of $o(\pi)$. In particular, $\S_{2m+1}^*$ is a subset of $\tilde\D_{2m+1}$.

\begin{lemma}\label{lem:cancel}
We have
\begin{align}
\label{cancel out-even}
&\sum_{\pi\in\S_{2m}}(-1)^{\widetilde\exc(\pi)}=\sum_{\pi\in\S_{2m}^*}(-1)^{\widetilde\exc(\pi)}=(-1)^{m-1}|\S_{2m}^*|,\\
\label{cancel out-odd}
&\sum_{\pi\in\tilde\D_{2m+1}}(-1)^{\widetilde\exc(\pi)}=\sum_{\pi\in\S_{2m+1}^*}(-1)^{\widetilde\exc(\pi)}=(-1)^{m}|\S_{2m+1}^*|.
\end{align}
\end{lemma}
\begin{proof}
For convenience, we use $A\triangle B$ to denote the symmetric difference of two sets $A$ and $B$. In particular, when $B=\{b\}$ is a singleton, we have
$$A\triangle\{b\}=\begin{cases}
A\setminus\{b\} & \text{if $b\in A$,}\\
A\cup\{b\} & \text{if $b\not\in A$.}
\end{cases}$$

Let us first deal with the even case. Take any $\pi\in\S_{2m}$ and set $\sigma=o(\pi)$. Note that the convention $\sigma(0)=\sigma(2m+1)=+\infty$ forces the union $\Dasc(\sigma)\cup\Ddes(\sigma)$ to be non-empty. Hence there are two possibilities for this union:
\begin{enumerate}[(i)]
  \item There exists some letter $x\neq m+1$ and $x\in\Dasc(\sigma)\cup\Ddes(\sigma)$.
  \item $\Dasc(\sigma)\cup\Ddes(\sigma)=\{m+1\}$, i.e., $\pi\in\S_{2m}^*$.
\end{enumerate}
If we are in case (i), let $x$ be the smallest such letter, then we see that
$$\Dasc(\varphi_x(\sigma))=\Dasc(\sigma)\triangle\{x\}, \text{ and } \Ddes(\varphi_x(\sigma))=\Ddes(\sigma)\triangle\{x\}.$$
Consequently, $\Dest(\varphi_x(\sigma))=\Dest(\sigma)\triangle\{x\}$, and
$\Exct(\psi_x(\pi))=\Exct(\pi)\triangle\{x\}$ by Lemma~\ref{f:transfor}.
Since $x\neq m+1$, we deduce that $\widetilde\exc(\pi)$ and $\widetilde\exc(\psi_x(\pi))$ have opposite parity, so they cancel each other in the first summation of~\eqref{cancel out-even}. All remaining permutations are in case (ii), thus we have proved the first equality in \eqref{cancel out-even}.

It is not difficult to see the following alternative description of $\S_{2m}^*$.

\begin{framed}
  {\bf A characterization of $\S_{2m}^*$:} Every permutation $\sigma$ in $o(\S_{2m}^*)$ can be uniquely constructed from an up-down permutation $\hat\sigma$ (i.e., $\hat\sigma(1)<\hat\sigma(2)>\hat\sigma(3)<\cdots<\hat\sigma(2m-2)>\hat\sigma(2m-1)$) consisting of letters from $[2m]\setminus\{m+1\}$, by inserting $m+1$ in one of the following ways:
  \begin{enumerate}[(e1)]
    \item before $\hat\sigma(1)$ if $m+1>\hat\sigma(1)$;
    \item after $\hat\sigma(2m-1)$ if $m+1>\hat\sigma(2m-1)$;
    \item inbetween $\hat\sigma(i)$ and $\hat\sigma(i+1)$ such that $\hat\sigma(i)>m+1>\hat\sigma(i+1)$;
    \item inbetween $\hat\sigma(i)$ and $\hat\sigma(i+1)$ such that $\hat\sigma(i)<m+1<\hat\sigma(i+1)$.
  \end{enumerate}
\end{framed}
Recall that $\widetilde\exc$ rejects $m+1$ as an excedance top, so inserting $m+1$ does not affect the value of $\widetilde\exc$ in view of Lemma~\ref{f:transfor}. Moreover, the excedance tops counted by $\widetilde\exc(o^{-1}(\hat\sigma))$ are precisely the peaks of $\hat\sigma$ (i.e., those $\hat\sigma(i)$ with $\hat\sigma(i-1)<\hat\sigma(i)>\hat\sigma(i+1)$, or to be precise, $\hat\sigma(2),\hat\sigma(4),\ldots,\hat\sigma(2m-2)$, since $\hat\sigma$ is an up-down permutation). So  $\widetilde\exc(\pi)=m-1$ for every $\pi\in\S_{2m}^*$, which takes care of the second equality in~\eqref{cancel out-even}.

Next for the odd case, again set $\sigma=o(\pi)$ for a given $\pi\in\tilde\D_{2m+1}$. New conventions $\sigma(0)=0$ and $\sigma(2m+2)=+\infty$ force the union $\Dasc(\sigma)\cup\Ddes(\sigma)$ to be non-empty. We proceed as in the even case, to find the smallest $x\in\Dasc(\sigma)\cup\Ddes(\sigma)\setminus\{m+1\}$ if it exists, and apply $\psi_x$ to explain the cancellations and establish the first equality of \eqref{cancel out-odd}. To show the second equality of \eqref{cancel out-odd}, we resort to the following equivalent description of $\S_{2m+1}^*$.
\begin{framed}
  {\bf A characterization of $\S_{2m+1}^*$:} Every permutation $\sigma$ in $o(\S_{2m+1}^*)$ can be uniquely constructed from a down-up permutation $\hat\sigma$ (i.e., $\hat\sigma(1)>\hat\sigma(2)<\hat\sigma(3)>\cdots<\hat\sigma(2m-1)>\hat\sigma(2m)$) consisting of letters from $[2m+1]\setminus\{m+1\}$, by inserting $m+1$ in one of the following ways:
  \begin{enumerate}[(o1)]
    \item before $\hat\sigma(1)$ if $m+1<\hat\sigma(1)$;
    \item after $\hat\sigma(2m)$ if $m+1>\hat\sigma(2m)$;
    \item inbetween $\hat\sigma(i)$ and $\hat\sigma(i+1)$ such that $\hat\sigma(i)>m+1>\hat\sigma(i+1)$;
    \item inbetween $\hat\sigma(i)$ and $\hat\sigma(i+1)$ such that $\hat\sigma(i)<m+1<\hat\sigma(i+1)$.
  \end{enumerate}
\end{framed}
  It follows from the above description and Lemma~\ref{f:transfor} that  $\widetilde\exc(\pi)=m$ for every $\pi\in\S_{2m+1}^*$. This completes the proof of the lemma.
\end{proof}

\subsection{Poupard numbers and the proof of Conjecture~\ref{conj:Chen3}}
In view of \eqref{tildeR}-\eqref{perR-odd} and  \eqref{cancel out-even}-\eqref{cancel out-odd}, we get
\begin{equation}
\label{perR=S*-even}
\per(R_{2m}) =|\S_{2m}^*|\ \ \text{and}\ \
\per(R_{2m+1}) =-|\S_{2m+1}^*|.
\end{equation}
Therefore, it remains to enumerate $\S_{2m}^*$ and $\S_{2m+1}^*$. It turns out that they are in simple bijections with certain subsets of alternating permutations which  were previously investigated by Foata and Han \cite{FH}, in their course of deriving new combinatorial interpretations for the finite difference equation system introduced by Christiane Poupard.

Let $\A_n$ denote the set of alternating (down-up) permutations. For any $\pi\in\A_n$, we suppose $\pi(i)=n$ for a certain $i$ ($1\le i\le n$) and use the convention $\pi(0)=\pi(n+1)=0$. Following Foata and Han~\cite{FH}, we introduce
$$\grn(\pi):=\max\{\pi(i-1),\pi(i+1)\},$$
and call it the {\it \textbf{gr}eater \textbf{n}eighbour} of $n$ in $\pi$. Also, let $\A_{n,k}:=\{\pi\in\A_n:\grn(\pi)=k\}$ for each $0\le k\le n-1$.

\begin{lemma}\label{lem:bij to Poupard}
For each $n\ge 1$, there exists a bijection $f:\S_{n}^*\to\A_{n+1,\lfloor\frac{n+1}{2}\rfloor}$. In particular, we have $|\S_{2m}^*|=|\A_{2m+1,m}|$ and $|\S_{2m+1}^*|=|\A_{2m+2,m+1}|$.
\end{lemma}
\begin{proof}
For the even case with $n=2m$, we refer to the boxed characterization of $\S_{2m}^*$ given in the proof of Lemma~\ref{lem:cancel}, and construct the image $f(\pi)$ for each $\pi\in\S_{2m}^*$ (set $\sigma=o(\pi)$) accordingly. We first apply the {\it complement map}
\begin{equation}\label{def:comp}
\tau=\tau(1)\cdots\tau(n)\mapsto \tau^{\mathrm{c}}:=(n+1-\tau(1))\cdots(n+1-\tau(n))
\end{equation}
 to $\sigma$, and then insert $2m+1$ as follows.
\begin{enumerate}[(e1)]
  \item If $\sigma^{\mathrm{c}}(1)=m$, then we insert $2m+1$ before $m$ to get a new permutation $f(\pi)$.
  \item If $\sigma^{\mathrm{c}}(2m)=m$, then we insert $2m+1$ after $m$ to get a new permutation $f(\pi)$.
  \item When $\sigma^{\mathrm{c}}(i)<\sigma^{\mathrm{c}}(i+1)=m<\sigma^{\mathrm{c}}(i+2)$, we insert $2m+1$ between $m$ and $\sigma^{\mathrm{c}}(i)$ to get a new permutation $f(\pi)$.
  \item When $\sigma^{\mathrm{c}}(i)>\sigma^{\mathrm{c}}(i+1)=m>\sigma^{\mathrm{c}}(i+2)$, we insert $2m+1$ between $m$ and $\sigma^{\mathrm{c}}(i+2)$ to get a new permutation $f(\pi)$.
\end{enumerate}
It is readily verified that in all the four cases, $f(\pi)$ is indeed a down-up permutation of length $2m+1$, whose greater neighbor of $2m+1$ is $m$. So $f$ is well-defined and we get its inverse  simply by removing $2m+1$ and then applying the complement map and $o^{-1}$.

The map $f$ for the odd case with $n=2m+1$ can be constructed similarly without applying complement map before we insert $2m+2$. The details are omitted.
\end{proof}

The last step towards proving Conjecture~\ref{conj:Chen3} is to utilize the bivariate generating function for all Poupard numbers $|\A_{n,k}|$. Set $g_n(k):=|\A_{2n-1,k-1}|~(n\ge 1,~1\le k\le2n-1)$, and $h_n(k):=|\A_{2n,k}|~(n\ge1,~1\le k\le 2n-1)$.
\begin{theorem}[Theorem~1.2 in \cite{FH}]
We have
\begin{equation}
\label{gf:Ztan}
1+\sum_{n\ge 1}\sum_{1\le k\le 2n+1}g_{n+1}(k)\frac{x^{2n+1-k}}{(2n+1-k)!}\cdot\frac{y^{k-1}}{(k-1)!}=\sec(x+y)\cos(x-y)
\end{equation}
and
\begin{equation}
\label{gf:Zsec}
1+\sum_{n\ge 1}\sum_{1\le k\le 2n+1}h_{n+1}(k)\frac{x^{2n+1-k}}{(2n+1-k)!}\cdot\frac{y^{k-1}}{(k-1)!}=\sec^2(x+y)\cos(x-y).
\end{equation}
\end{theorem}

\begin{lemma}\label{lem:central Poupard}
For each $n=0,1,2,\ldots$, we have
\begin{equation}\label{secant}
g_{n+1}(n+1)=\sum_{k=0}^n\binom{n}{k}E_{2k}
\end{equation}
and
\begin{equation}\label{tangent}
h_{n+1}(n+1)=\sum_{k=0}^n\binom{n}{k}E_{2k+1}.
\end{equation}
\end{lemma}
\begin{proof} Observe that
\begin{align*}
\sec(x+y)\cos(x-y)=&\biggl(\sum_{k\geq0}E_{2k}\frac{(x+y)^{2k}}{(2k)!}\biggr)
\biggl(\sum_{k\geq0}(-1)^k\frac{(x-y)^{2k}}{(2k)!}\biggr)
\\=&\sum_{n=0}^\infty\sum_{k=0}^n(-1)^{n-k}\frac{E_{2k}(x+y)^{2k}}{(2k)!}\cdot\frac{(x-y)^{2n-2k}}{(2n-2k)!}.
\end{align*}
Combining this with \eqref{gf:Ztan}, we get
$$g_{n+1}(n+1)=\sum_{k=0}^n(-1)^kE_{2k}\sum_{i=0}^{2k}\binom{n}{i}\binom{n}{2k-i}(-1)^{i},$$
which gives~\eqref{secant} since
$$
(-1)^k\binom{n}{k}=[x^{2k}](1-x^2)^n=[x^{2k}](1-x)^n(1+x)^n=\sum_{i=0}^{2k}\binom{n}{i}\binom{n}{2k-i}(-1)^{i},
$$ where $[x^m]f(x)$ denotes the coefficient of $x^m$ in the power series expansion of $f(x)$.
The second formula \eqref{tangent} follows from the same manipulation by noticing
$\tan'x=\sec^2 x$.
\end{proof}

With all the needed pieces on hand, Conjecture~\ref{conj:Chen3} now follows from \eqref{perR=S*-even}, and Lemmas~\ref{lem:bij to Poupard} and \ref{lem:central Poupard}.

\subsection{Bala's continued fraction conjecture}\label{subsec:Bala}
Recall that the {\em descent number} of a word $w=w_1w_2\cdots w_n$ over $\N$ is
$$\des(w):=|\{i\in[n-1]:w_i>w_{i+1}\}|.
$$
The classical {\em Eulerian polynomial} $A_n(t)$ may be defined by Euler's formula
\begin{equation}\label{euler}
\frac{A_n(t)}{(1-t)^{n+1}}=\sum_{k\geq0}(k+1)^nt^k.
\end{equation}
It is well known that $A_n(t)$ has the following two interpretations (see~\cite[Chap.~1]{St}):
$$
\sum_{\pi\in\S_n}t^{\des(\pi)}=A_n(t)=\sum_{\pi\in\S_n}t^{\exc(\pi)}.
$$
Consider a variation of the Eulerian polynomials:
 $\tilde A_n(t):=\sum_{\pi\in\S_{n}}t^{\widetilde\exc(\pi)}$ for $n\ge1$.
 For convenience, we list the first few terms of $\tilde A_n(t)$ as follows:
\begin{align*}
\tilde A_1(t)&=1,\\
\tilde A_2(t)&=2,\\
\tilde A_3(t)&=2+4t,\\
\tilde A_4(t)&=4+16t+4t^2,\\
\tilde A_5(t)&=4+48t+60t^2+8t^3,\\
\tilde A_6(t)&=8+160t+384t^2+160t^3+8t^4,\\
\tilde A_7(t)&=8+368t+1952t^2+2176t^3+520t^4+16t^5,\\
\tilde A_8(t)&=16+1152t+9648t^2+18688t^3+9648t^4+1152t^5+16t^6.
\end{align*}

We have the following result for $\tilde A_{2n}(t)$.

\begin{theorem}\label{2:eulerian} Let $\S_n^{(2)}$ be the set of all permutations of the multiset $\{1,1,2,2,\ldots,n,n\}$. Then
\begin{equation}\label{con:2euler}
\tilde A_{2n}(t)=2^n\sum_{\pi\in\S_n^{(2)}} t^{\des(\pi)}.
\end{equation}
\end{theorem}

 The polynomial $\sum_{\pi\in\S_n^{(2)}} t^{\des(\pi)}$ that we denote by $A_n^{(2)}(t)$ is called the $n$th {\em $2$-Eulerian polynomial}. Analog to Euler's formula~\eqref{euler}, MacMahon~\cite[Volume 2, p.~211]{Macm}  proved that
\begin{equation}\label{macmahon}
\frac{A_n^{(2)}(t)}{(1-t)^{2n+1}}=\sum_{k\geq 0}\binom{k+2}{2}^n t^k.
\end{equation}
Many interesting properties of the $2$-Eulerian polynomials have been extensively studied ever since~\cite{CH,Ardila,LMMZ}. Remarkably, Ardila~\cite{Ardila} proved that $2$-Eulerian polynomials are the $h$-polynomials of the dual bipermutahedron.

To proof Theorem~\ref{2:eulerian}, we need two lemmas.

\begin{lemma}\label{lema:Enk}
Let $\tilde A(1,0):=1$ and $\tilde A(n,k):=|\{\pi\in\S_n:\widetilde\exc(\pi)=k\}|$ for $n\geq2$ and $0\leq k\leq n-2$. Then $\tilde A(n,k)$ satisfies the following recurrence relation:
$$
\tilde A(n,k)=
\begin{cases}
(n-k-1) \tilde A(n-1,k-1)+(k+2) \tilde A(n-1,k),&\quad\text{if $n$ is even;}\\
(n-k) \tilde A(n-1,k-1)+(k+1) \tilde A(n-1,k),&\quad\text{if $n$ is odd.}
\end{cases}
$$
Consequently,
\begin{align*}
\tilde A(2n,k)=&\ (k+2)(k+1)\tilde A(2n-2,k)+(2n-k-1)(2k+2)\tilde A(2n-2,k-1)\\
&\ +(2n-k)(2n-k-1) \tilde A(2n-2,k-2).
\end{align*}
\end{lemma}

\begin{proof}
Consider a permutation $\pi\in\S_{n-1}$ in cycle form (i.e., write $\pi$ as a product of disjoint cycles). For any $a\in[n-1]$ with $\pi(a)\neq\lceil\frac{n}{2}\rceil$, we call the slot inbetween $a$ and $\pi(a)$ a {\em decreasing (resp.~increasing) slot} if $a\geq\pi(a)$ (resp.~$a<\pi(a)$). Clearly, if $\widetilde\exc(\pi)=i$, then $\pi$ has $i$ increasing slots and $n-2-i$ decreasing slots.

If $n\ge 3$ is odd, then
a permutation counted by $\tilde A(n,k)$ can be constructed
\begin{itemize}
\item either from $\pi$ with $\widetilde\exc(\pi)=k-1$ by inserting $n$ into one of its $n-k-1$ decreasing slots or just before $\lceil n/2\rceil$,
\item or from $\pi$ with $\widetilde\exc(\pi)=k$ by  inserting $n$ into one of its $k$ increasing slots or setting $n$ as a $1$-cycle.
\end{itemize}
When $n$ is even, a permutation counted by $\tilde A(n,k)$ can be constructed
\begin{itemize}
\item either from $\pi$ with $\widetilde\exc(\pi)=k-1$ by inserting $0$ into one of its $n-k-1$ decreasing slots and then increasing all letters by $1$,
\item or from $\pi$ with $\widetilde\exc(\pi)=k$ in one of the following three ways and then increasing all letters by $1$:
\begin{enumerate}
\item  inserting $0$ into one of its $k$ increasing slots;
\item  inserting $0$ just before $\lceil n/2\rceil$;
\item setting $0$ as a $1$-cycle.
\end{enumerate}
\end{itemize}
This proves the desired recurrence relation for $\tilde A(n,k)$.
\end{proof}

\begin{lemma}\label{lema:Ank}
Write $A_n^{(2)}(t)=\sum_{k=0}^{2n-2}A^{(2)}(n,k)t^k$. Then
\begin{align*}
A^{(2)}(n,k)=&\ \binom{k+2}{2}A^{(2)}(n-1,k)+(2n-k-1)(k+1)A^{(2)}(n-1,k-1)\\
&\ +\binom{2n-k}{2}A^{(2)}(n-1,k-2).
\end{align*}
\end{lemma}

\begin{proof}
Applying MacMahon's formula~\eqref{macmahon}, we have
\begin{align*}
\frac{A_n^{(2)}(t)}{(1-t)^{2n+1}}&=\sum_{k\geq 0}\binom{k+2}{2}^n t^k=\sum_{k\geq 0}\binom{k+2}{2}^{n-1}(k(k-1)/2+2k+1) t^k\\
&=\frac{t^2}{2}(A_{n-1}^{(2)}(t)(1-t)^{-2n+1})''+2t(A_{n-1}^{(2)}(t)(1-t)^{-2n+1})'+A_{n-1}^{(2)}(t)(1-t)^{-2n+1}.
\end{align*}
Multiplying both sides by $(1-t)^{2n+1}$ gives
\begin{align*}
A_n^{(2)}(t)&=(A_{n-1}^{(2)}(t))''t^2(1-t)^2/2+(A_{n-1}^{(2)}(t))'t(1-t)((2n-3)t+2)\\
&\quad+A_{n-1}^{(2)}(t)(2n^2t^2-5nt^2+3t^2+4nt-4t+1).
\end{align*}
Extracting the coefficients of $t^k$ from both sides of the above equation yields the desired recurrence relation for $A^{(2)}(n,k)$.
\end{proof}

\noindent{\bf Proof of Theorem 4.7}. Comparing the recurrence relation for $\tilde A(2n,k)$ in Lemma~\ref{lema:Enk} with that for $A^{(2)}(n,k)$ in Lemma~\ref{lema:Ank}, we get the desired \eqref{con:2euler}. \qed

The following fundamental generating function is known as {\em Touchard's
continued fraction}~\cite{Tou} (see also~\cite{Prod}):
$$
\sum_{k\geq0}q^{\binom{k+1}{2}}z^k=\frac{1}{1-z+\displaystyle\frac{(1-q)z}{1-z+\displaystyle\frac{(1-q^2)z}{1-z+\displaystyle\frac{(1-q^3)z}{\cdots}}}}.
$$

\begin{theorem}\label{thm:2-euler:con}
The exponential generating function for $A_n^{(2)}(t)$ has the following continued fraction expansion:
\begin{equation}\label{2-euler:con}
\sum_{n\geq0}\frac{tA_n^{(2)}(t)}{n!}z^n=t-1+\frac{1-t}{1-t+\displaystyle\frac{(1-e^{(1-t)^2z})t}{1-t+\displaystyle\frac{(1-e^{2(1-t)^2z})t}{1-t+\displaystyle\frac{(1-e^{3(1-t)^2z})t}{\cdots}}}}.
\end{equation}
\end{theorem}
\begin{proof}
By MacMahon's formula~\eqref{macmahon}, we have
$$
tA_n^{(2)}(t)=(1-t)^{2n+1}\sum_{k\geq 0}\binom{k+2}{2}^n t^{k+1}.
$$
Multiplying both sides by $\frac{z^n}{n!}$ and summing over all $n\geq0$ gives
\begin{align*}
\sum_{n\geq0}\frac{tA_n^{(2)}(t)}{n!}z^n&=\sum_{n\geq0}\frac{(1-t)^{2n+1}z^n}{n!}\sum_{k\geq 0}\binom{k+2}{2}^n t^{k+1}\\
&=(1-t)\sum_{k\geq0}t^{k+1}\sum_{n\geq0}\frac{(\binom{k+2}{2}(1-t)^2z)^n}{n!}\\
&=(1-t)\sum_{k\geq0}t^{k+1}(e^{(1-t)^2z})^{\binom{k+2}{2}}.
\end{align*}
Applying Touchard's continued fraction yields~\eqref{2-euler:con}.
\end{proof}

The number $2^{-n}\sum_{k=0}^n\binom{n}{k}E_{2k}$ is called {\em generalized Euler number of type $2^n$} as in \cite[A005799]{oeis}. Combining~\eqref{bino:chen}, \eqref{tildeR}, \eqref{perR-even} and~\eqref{con:2euler}, we get
\begin{equation}\label{sign:euler2}
(-1)^{n+1}A_n^{(2)}(-1)=2^{-n}\sum_{k=0}^n\binom{n}{k}E_{2k},\quad\text{for $n\geq1$}.
\end{equation}
It then follows from Theorem~\ref{thm:2-euler:con} the following exponential generating function for $2^{-n}\sum_{k=0}^n\binom{n}{k}E_{2k}$, which was conjectured by Peter Bala (2019) in \cite[A005799]{oeis}.

\begin{corollary} We have
\begin{equation*}
\sum_{n\geq0}\frac{2^{-n}\sum_{k=0}^n\binom{n}{k}E_{2k}}{n!}z^n=\frac{2}{2-\displaystyle\frac{1-e^{-4z}}{2-\displaystyle\frac{1-e^{-8z}}{2-\displaystyle\frac{1-e^{-12z}}{\cdots}}}}.
\end{equation*}
\end{corollary}

A permutation $\pi\in\S_n^{(2)}$ is {\em alternating} if
$$
\pi(1)\leq \pi(2)>\pi(3)\leq \pi(4)>\pi(5)\leq \cdots.
$$
Let $\mathrm{Alt}_n$ be the set of all alternating permutations in $\S_n^{(2)}$. On one hand, Gessel~\cite[Eq.~(6)]{Ges} proved that
$$
|\mathrm{Alt}_n|=2^{-n}\sum_{k=0}^n\binom{n}{k}E_{2k}.
$$
On the other hand, Lin--Ma--Ma--Zhou~\cite[Corollary~2.15]{LMMZ} interpreted the $\gamma$-coefficients $\gamma_{n,k}^{(2)}$ appearing in the expansion (see the survey of Athanasiadis~\cite{Ath} on the theme of $\gamma$-positivity)
\begin{equation}\label{gamma:2}
A_n^{(2)}(t)=\sum_{k=0}^{n-1}\gamma_{n,k}^{(2)}t^k(1+t)^{2(n-1)-2k}
\end{equation}
as some class of {\em weakly increasing trees} (see~\cite{LMMZ} for the definition). It follows from the above expansion and Eq.~\eqref{sign:euler2} that
$$
\gamma_{n,n-1}^{(2)}=2^{-n}\sum_{k=0}^n\binom{n}{k}E_{2k}.
$$
Thus, we have the following interesting equinumerosity.

\begin{corollary}\label{cor:trees}
The number of alternating permutations in $\S_n^{(2)}$ equals the number of weakly increasing trees on $\{1,1,2,2,\ldots,n-1,n-1,n\}$ with $n$ leaves and without young leaves.
\end{corollary}

It would be interesting to find a bijective proof of Corollary~\ref{cor:trees}, which would provide an alternative approach to the even case of Conjecture~\ref{conj:Chen3}.

\subsection{Combinatorics of the \texorpdfstring{$\gamma$}{gamma}-positivity of \texorpdfstring{$\tilde A_{2m}(t)$}{}}
The rest of this section is devoted to a group action proof of the $\gamma$-positivity of $\tilde A_{2m}(t)$ that results in a new interpretation of $\gamma_{n,k}^{(2)}$ defined in~\eqref{gamma:2}.

For any $\pi\in\S_{2m}$, note that $|\Dest(\pi)|=\des(\pi)$. Introduce a variant of descents of $\pi$:
$$
\widetilde\des(\pi):=|\Dest(\pi)\setminus\{m+1\}|.
$$
By Lemma~\ref{f:transfor}, we have
$$
\tilde A_{2m}(t)=\sum_{\pi\in\S_{2m}} t^{\widetilde \des(\pi)}.
$$
With the convention $\pi(0)=\pi(2m+1)=+\infty$, for $i\in[2m]$, the letter $\pi(i)$ is called a {\em valley} (resp.~{\em peak})  of $\pi$ if $\pi(i-1)>\pi(i)<\pi(i+1)$ (resp.~$\pi(i-1)<\pi(i)>\pi(i+1)$). Denote by $\Val(\pi)$ and $\Peak(\pi)$ the set of valleys and the set of peaks of $\pi$, respectively.

\begin{theorem}\label{thm:gam}
The polynomial $\tilde A_{2m}(t)$ has the $\gamma$-positive expansion
\begin{equation}
\tilde A_{2m}(t)=\sum_{\pi\in\S_{2m}} t^{\widetilde \des(\pi)}=\sum_{k=0}^{m-1}|\tilde D_{2m,k}|t^k(1+t)^{2m-2-2k},
\end{equation}
where
$$
\tilde D_{2m,k}:=\{\pi\in\S_{2m}: \Ddes(\pi)\setminus\{m+1\}=\emptyset, m+1\notin\Val(\pi)\text{ and $\widetilde\des(\pi)=k$}\}.
$$
\end{theorem}

\begin{example}
For instance,  we have
$$
\tilde D_{4,0}=\{1234,3124,1324,2314\}\text{ and }\tilde D_{4,1}=\{1342,1423,1432,2341,2413,2431,3142,3241\},
$$
and so
$$
\tilde A_{4}(t)=4(1+t)^2+8t.
$$
\end{example}

An immediate consequence of Theorems~\ref{2:eulerian} and~\ref{thm:gam} is the following permutation interpretation of $\gamma_{n,k}^{(2)}$ (see~\cite[Corollary~2.15]{LMMZ} for another interpretation of $\gamma_{n,k}^{(2)}$ in terms of trees).
\begin{corollary}
Let  $\gamma_{n,k}^{(2)}$ be defined in~\eqref{gamma:2}. Then $\gamma_{n,k}^{(2)}=2^{-n}|\tilde D_{2n,k}|$ for $n\geq1$.
\end{corollary}

In order to prove Theorem~\ref{thm:gam}, we need the following interesting equidistribution.

\begin{lemma}\label{bij:fvc}
There exists a bijection $\eta$ preserving the number of descents  between
$$
\mathcal{P}_m:=\{\pi\in\S_{2m}: \Ddes(\pi)=\emptyset, m+1\text{ is a peak}\}
$$
and
$$
\mathcal{V}_m:=\{\pi\in\S_{2m}: \Ddes(\pi)=\emptyset, m+1\text{ is a valley}\}.
$$
\end{lemma}

The proof of the above lemma can be considered as a nice application of the classical  {\em Fran\c{c}on--Viennot bijection}~\cite{Fran} that encodes permutations as Laguerre histories.

Recall that a {\em Motzkin path} of length $n$ is a lattice path in the first  quadrant starting from $(0,0)$, ending at $(n,0)$, and using three possible steps:
 $$
 U=(1,1) \text{ (up step), } L=(1,0) \text{ (level step), and } D=(1,-1) \text{ (down step).}
 $$
A Motzkin path in which each level step is further  distinguished into two different types  $L_0$ (in blue) and $L_1$ (in red) is called a {\em$2$-Motzkin path}.
 Thus, each $2$-Motzkin path can be represented as a word over the alphabet $\{U,D,L_0,L_1\}$. A {\em Laguerre history} of length $n$ is a pair $(w,\mu)$, where $w=w_1\cdots w_n$ is a $2$-Motzkin path and $\mu=(\mu_1,\cdots,\mu_n)$ is a sequence of weights satisfying $0\leq\mu_i\leq h_i(w)$, and $h_i(w)$ denotes  the {\em height} of the starting point of the $i$-th step of $w$. Denote by $\Lag_n$ the set of all Laguerre histories of length $n$.

 Using the convention $\pi(0)=\pi(n+1)=0$, the  Fran\c{c}on--Viennot bijection  $\phi_{FV}:\S_n\rightarrow\Lag_{n-1}$ that we need  is the following modified version (see~\cite{LY}) defined as $\phi_{FV}(\pi)=(w,\mu)\in\Lag_{n-1}$, where for each $k\in[n]$ with $i=\pi(k)\leq n-1$:
$$
  w_i=\left\{
  \begin{array}{ll}
  U& \mbox{if $\pi(k-1)>\pi(k)=i<\pi(k+1)$},  \\
 D&\mbox{if $\pi(k-1)<\pi(k)=i>\pi(k+1)$},  \\
   L_0&\mbox{if $\pi(k-1)<\pi(k)=i<\pi(k+1)$},\\
  L_1 &\mbox{if $\pi(k-1)>\pi(k)=i>\pi(k+1)$},
  \end{array}
  \right.
  $$
  and $\mu_i$ is the number of $\pp$ patterns with $i$ representing the $2$, i.e.,
$$\mu_i=\pp_i(\pi):=|\{j: j-1>k\text{ and } \pi(j-1)<\pi(k)=i<\pi(j)\}|.$$
For example, if $\pi=21637548\in\S_8$, then $\phi_{FV}(\pi)=(w,\mu)$, where $w=UDUUL_1DD$ and $\mu=(0,1,0,0,1,2,1)$; see Fig.~\ref{fig:Theta} below for an illustration, where the $L_1$ step is colored red. It was known that $\phi_{FV}$ is a bijection and the reader is referred to~\cite{LY} for a recursive description of its inverse $\phi_{FV}^{-1}$.

\begin{proof}[{\bf Proof of Lemma~\ref{bij:fvc}}] Recall the complement map $\pi\mapsto \pi^{\mathrm{c}}$ defined in~\eqref{def:comp}. It is convenient to consider the two sets
$$
\mathcal{P}_m^{\mathrm{c}}:=\{\pi^{\mathrm{c}}: \pi\in \mathcal{P}_m\}\quad\text{and}\quad \mathcal{V}_m^{\mathrm{c}}:=\{\pi^{\mathrm{c}}: \pi\in \mathcal{V}_m\}.
$$
Note that the letter $m+1$ in a permutation $\pi\in \mathcal{P}_m$ becomes the letter $m$ in $\pi^{\mathrm{c}}$. We aim to construct a bijection $\eta': \mathcal{P}_m^{\mathrm{c}}\rightarrow \mathcal{V}_m^{\mathrm{c}}$ that preserves the number of ascents of permutations and then set $\eta$ to be the map $\pi\mapsto (\eta'(\pi^{\mathrm{c}}))^{\mathrm{c}}$. To do this, we will introduce an involution $\Theta$ on $\Lag_{2m-1}$.

Let $(w,\mu)\in\Lag_{2m-1}$ be a Laguerre history. For any up step $w_i=U$, there is a unique down step $w_{i'}=D$ to the right of $w_i$, closest to $w_i$,  whose ending point has the same height as the starting point of $w_i$. We call $w_{i'}$ (resp.~$w_i$) the associated down (resp.~up) step of $w_i$ (resp.~$w_{i'}$). For example, for the Laguerre history in the left part of Fig.~\ref{fig:Theta}, the associated down step of $w_4$ is $w_6$ and the remaining two associated pairs are $(w_1,w_2)$ and $(w_3,w_7)$. Now define $\Theta(w,\mu)=(w^*,\mu^*)$, where
$$
w_{2m-i}^*=\begin{cases}
U & \text{if $w_i=D$,}\\
D & \text{if $w_i=U$,}\\
w_i & \text{if $w_i$ is a level step,}
\end{cases}
$$
and
$$
\mu_{2m-i}^*=
\begin{cases}
\mu_{i'} \quad&\text{if $w_i=U$ (resp.~$w_i=D$) whose associated down (resp.~up) step is $w_{i'}$},\\
\mu_{i} &\text{if $w_i$ is a level step}.
\end{cases}
$$
See Fig.~\ref{fig:Theta} for an instance of $\Theta$, where $w=UDUUL_1DD$ and $\mu=(0,1,0,0,1,2,1)$.

\begin{figure}[ht]
\centering
\begin{tikzpicture}[scale=0.28]
\draw[thick] (0,0) to (2,2);\node at (0,0) {$\bullet$};
\draw[thick] (2,2) to (4,0);\node at (4,0) {$\bullet$};\node at (2,2) {$\bullet$};
\draw[thick] (4,0) to (6,2);\draw[thick] (6,2) to (8,4);\node at (6,2) {$\bullet$};
\node at (8,4) {$\bullet$};\draw[-,color=red] (8,4) to (11,4);\node at (11,4) {$\bullet$};
\draw[thick] (11,4) to (15,0);\node at (13,2) {$\bullet$};\node at (15,0) {$\bullet$};
\node at (0.7,1.5) {$0$};\node at (3.3,1.5) {$1$};
\node at (4.8,1.5) {$0$}; \node at (6.6,3.3) {$0$};
\node at (9.5,4.8) {$1$}; \node at (12.4,3.4) {$2$};
\node at (14.4,1.5) {$1$};

\node at (18,2) {$\longrightarrow$};

\node at (18,3) {$\Theta$};

\draw[dashed] (7,3) to (12,3);

\draw[thick] (21,0) to (25,4);\node at (21,0) {$\bullet$};\node at (23,2) {$\bullet$};\node at (25,4) {$\bullet$};
\draw[-,color=red] (25,4) to (28,4);\node at (28,4) {$\bullet$};
\draw[thick] (28,4) to (32,0);\node at (30,2) {$\bullet$};\node at (32,0) {$\bullet$};
\draw[thick] (32,0) to (34,2);\draw[thick] (34,2) to (36,0);\node at (34,2) {$\bullet$};\node at (36,0) {$\bullet$};

\node at (21.7,1.5) {$0$};\node at (23.7,3.5) {$0$};
\node at (26.5,4.8) {$1$};\node at (29.2,3.5) {$2$};
\node at (31.2,1.5) {$1$};\node at (32.8,1.5) {$0$};\node at (35.3,1.5) {$1$};
\end{tikzpicture}
\caption{The construction of $\Theta$: a red level step represents $L_1$. \label{fig:Theta}}
\end{figure}

It is clear from the construction that $\Theta: (w,\mu)\mapsto (w^*,\mu^*)$ is an involution on $\Lag_{2m-1}$ for which
$$
\text{the $m$-th step of $w$ is an up step}\iff\text{the $m$-th step of $w^*$ is a down step}.
$$
It is routine to check that $\eta':=\phi_{FV}^{-1}\circ\Theta\circ\phi_{FV}$ is a bijection between $\mathcal{P}_m^c$ and $\mathcal{V}_m^c$ preserving the number of ascents.
\end{proof}

\begin{proof}[{\bf Proof of Theorem~\ref{thm:gam}}]  Recall the Foata--Strehl action $\varphi_x$ defined in~\eqref{def:FSact}. Br\"and\'en~\cite{Bra} modified the Foata--Strehl action $\varphi_x$ as
$$
\varphi_x'(\pi)=
\begin{cases}
\varphi_x(\pi),\quad&\text{if $x$ is a double ascent/descent value of $\pi$};\\
\pi, &\text{otherwise}.
\end{cases}
$$
For our purpose, for any $x\in[2m]$ we consider the restricted version of $\varphi_x'$ defined by
 $$
\widetilde\varphi_x(\pi)=
\begin{cases}
\varphi_x'(\pi)\quad&\text{if $x\neq m+1$},\\
\pi \quad&\text{if $x=m+1$},
\end{cases}
$$
 where $\pi\in\S_{2m}$. For our discussions below, the reader is advised to envision the restricted action $\widetilde\varphi_x$ using the so-called valley-hopping interpretation as in Fig.~\ref{FS:action}.

Since $\varphi_x'$'s are involutions on $\S_{2m}$ and they commute, so do $\widetilde\varphi_x$'s. Thus, for any subset $S\subseteq[2m]$, we can define the mapping $\widetilde\varphi_S: \S_{2m}\rightarrow\S_{2m}$ by
$$
\widetilde\varphi_S:=\prod_{x\in S}\widetilde\varphi_x.
$$
This induces a $\Z_2^{2m}$-action on $\S_{2m}$ via the mappings  $\widetilde\varphi_S$ with $S\subseteq[2m]$.
Let $\Orb(\pi)$ be the orbit of $\pi$ under this action. If $x\in[2m]$ is a double descent/ascent value of $\pi$ different from $m+1$, then $x$ becomes a double ascent/descent value of $\widetilde\varphi_x(\pi)$ and so
$$
\widetilde\des(\widetilde\varphi_x(\pi))=
\begin{cases}
\widetilde\des(\pi)+1\quad\text{if $x\neq m+1$ is a double ascent value of $\pi$},\\
\widetilde\des(\pi)-1\quad\text{if $x\neq m+1$ is a double descent value of $\pi$}.
\end{cases}
$$
Therefore, if we use $\hat\pi$ to denote the unique permutation in $\Orb(\pi)$ with $\Ddes(\hat\pi)\setminus\{m+1\}=\emptyset$, then
$$
\sum_{\sigma\in\Orb(\pi)}t^{\widetilde\des(\sigma)}=t^{\widetilde\des(\hat\pi)}(1+t)^{|\Dasc(\hat\pi)\setminus\{m+1\}|}.
$$
Now we need to consider the following two cases.
\begin{itemize}
\item If $m+1$ is a double descent or a double ascent value of $\hat\pi$, then $|\Dasc(\hat\pi)\setminus\{m+1\}|=2m-2-2\widetilde\des(\hat\pi)$ and so in this case
$$
\sum_{\sigma\in\Orb(\pi)}t^{\widetilde\des(\sigma)}=t^{\widetilde\des(\hat\pi)}(1+t)^{2m-2-2\widetilde\des(\hat\pi)}.
$$
\item If $m+1$ is a peak of $\hat\pi$, then $|\Dasc(\hat\pi)\setminus\{m+1\}|=2m-3-2\widetilde\des(\hat\pi)$. On the other hand,  $m+1$ is a valley of $\eta(\hat\pi)$ by Lemma~\ref{bij:fvc}, $\Ddes(\eta(\hat\pi))\setminus\{m+1\}=\emptyset$,
$\widetilde\des(\eta(\hat\pi))=\widetilde\des(\hat\pi)+1$
and $$ |\Dasc(\eta(\hat\pi))\setminus\{m+1\}|=2m-1-2\widetilde\des(\eta(\hat\pi)).
$$
Thus, we have
\begin{align*}
\sum_{\sigma\in\Orb(\pi)\biguplus\Orb(\eta(\hat\pi))}t^{\widetilde\des(\sigma)}&=\sum_{\sigma\in\Orb(\pi)}t^{\widetilde\des(\sigma)}+\sum_{\sigma\in\Orb(\eta(\hat\pi))}t^{\widetilde\des(\sigma)}\\
&=t^{\widetilde\des(\hat\pi)}(1+t)^{2m-3-2\widetilde\des(\hat\pi)}+t^{\widetilde\des(\hat\pi)+1}(1+t)^{2m-3-2\widetilde\des(\hat\pi)}\\
&=t^{\widetilde\des(\hat\pi)}(1+t)^{2m-2-2\widetilde\des(\hat\pi)}.
\end{align*}
\end{itemize}
Combining the above two cases, we obtain the desired $\gamma$-positive expansion for $\tilde A_{2m}(t)$.
\end{proof}

\section{Concluding remarks}

In this paper, we study permanents of the floor function of some fractions and the sign function of some trigonometric functions and establish their intriguing connections with several classical combinatorial sequences. It would be interesting to investigate the combinatorics of permanents of the floor function or the sign function of other elementary functions.

In the course of proving several permanent conjectures, we introduce the crucial  action $\phi_{k,l}$ in Definition~\ref{Def:action} on matrices that preserves  the permanents. For any $S\subseteq[n]\times[n]$, define the transformation matrix $T_S$ with respect to $S$ by
$$
\biggl(\prod_{(k,l)\in S}\phi_{k,l}\biggr)(A)=A\circ T_S,
$$
where $A$ is any $n\times n$ matrix over  $\mathbb{R}$. Let us consider the set of  transformation matrices
$$\mathcal{T}_n:=\{T_S: S\subseteq[n]\times[n]\}.$$
For instance, $\mathcal{T}_2$ consists of four matrices
$$
\begin{bmatrix}
1 & 1 \\
1 & 1
\end{bmatrix},\,\,
\begin{bmatrix}
1 & -1 \\
-1 & 1
\end{bmatrix},\,\,
\begin{bmatrix}
-1 & 1 \\
1 & -1
\end{bmatrix},\,\,
\begin{bmatrix}
-1 & -1 \\
-1 & -1
\end{bmatrix}.
$$

\begin{proposition}\label{char-action}
For any positive integer $n$, the transformation group $\mathcal{T}_n$ is isomorphic to
the group $\Z_2^{2n-2}$, where $\Z_2=\Z/2\Z$. Consequently,
$|\mathcal{T}_n|=2^{2n-2}$.
\end{proposition}

\begin{proof}[{\bf Sketch of the proof}]
Under the Hadamard product, $\mathcal{T}_n$ forms an abelian group whose non-identity elements
always have order $2$. By the well-known Structure Theorem for Finite Abelian Groups,   $\mathcal{T}_n\cong\Z_2^{\ell(n)}$ for some $\ell(n)\in\N$. It remains to show that $\ell(n)=2n-2$. This will be done once we can show that there are exactly $2n-2$ generators of $\mathcal{T}_n$.

For any $S\subseteq[n]\times[n]$, we define $\phi_S:=\prod_{(k,l)\in S}\phi_{k,l}$,
and call  $S$ a {\em kernel} if $\phi_S$ is the identity action. If $S$ is a kernel, then for any $(a,b)\in S$ we have $\phi_{(a,b)}=\phi_{S\setminus\{(a,b)\}}$. Consider the subset $\mathsf{G}_n=\mathsf{L}_n\uplus \mathsf{R}_n$, the disjoint union of
$$
\mathsf{L}_n:=\{(i,i), (i+1,i): 1\leq i\leq \lfloor n/2\rfloor\}\quad\text{and}\quad \mathsf{R}_n:=\{(i-1,i), (i,i): \lfloor n/2\rfloor+1<i\leq n\}.
$$
  Both $\mathsf{L}_n$ and $\mathsf{R}_n$ form (nearly) half of the border strip around the diagonal of the $n\times n$ grid, and $|\mathsf{G}_n|=2n-2$.
See Fig.~\ref{G_n} for examples of $\mathsf{G}_n$ with $n=7,8$.

\begin{figure}[ht]
\begin{center}
\begin{tikzpicture}[scale=.5]

\draw[step=1,color=gray] (0,0) grid (7,7);
\draw[step=1,color=gray] (10,0) grid (18,8);

\draw(.5,6.5) node{\blue{\circle*{4}}};\draw(.5,5.5) node{\blue{\circle*{4}}};
\draw(1.5,5.5) node{\blue{\circle*{4}}};\draw(1.5,4.5) node{\blue{\circle*{4}}};
\draw(2.5,4.5) node{\blue{\circle*{4}}};\draw(2.5,3.5) node{\blue{\circle*{4}}};
\draw(4.5,3.5) node{\red{\circle*{4}}};\draw(4.5,2.5) node{\red{\circle*{4}}};
\draw(5.5,2.5) node{\red{\circle*{4}}};\draw(5.5,1.5) node{\red{\circle*{4}}};
\draw(6.5,1.5) node{\red{\circle*{4}}};\draw(6.5,.5) node{\red{\circle*{4}}};

\draw(10.5,7.5) node{\blue{\circle*{4}}};\draw(10.5,6.5) node{\blue{\circle*{4}}};
\draw(11.5,6.5) node{\blue{\circle*{4}}};\draw(11.5,5.5) node{\blue{\circle*{4}}};
\draw(12.5,5.5) node{\blue{\circle*{4}}};\draw(12.5,4.5) node{\blue{\circle*{4}}};
\draw(13.5,4.5) node{\blue{\circle*{4}}};\draw(13.5,3.5) node{\blue{\circle*{4}}};
\draw(15.5,3.5) node{\red{\circle*{4}}};\draw(15.5,2.5) node{\red{\circle*{4}}};
\draw(16.5,2.5) node{\red{\circle*{4}}};\draw(16.5,1.5) node{\red{\circle*{4}}};
\draw(17.5,1.5) node{\red{\circle*{4}}};\draw(17.5,.5) node{\red{\circle*{4}}};

\end{tikzpicture}
\end{center}
\caption{Examples of $\mathsf{G}_n$ for $n=7,8$, where  $\mathsf{L}_n$ and $\mathsf{R}_n$ correspond  respectively to the blue and the red dots.\label{G_n}}
\end{figure}

Now we claim that the elements in $\{\phi_{(a,b)}: (a,b)\in \mathsf{G}_n\}$ form a set of generators for all the actions in $\{\phi_S:S\subseteq[n]\times[n]\}$, which would complete the proof. To prove this, one needs to check:
\begin{itemize}
\item The elements in $\{\phi_{(a,b)}: (a,b)\in \mathsf{G}_n\}$  generate  all $\phi_{(i,j)}$ for $(i,j)\in[n]\times[n]$. To see  this, note that the set $\{(i,j),(i+1,j),(i,j+1),(i+1,j+1)\}$ is a kernel for any $(i,j)\in[n-1]\times[n-1]$ and  the set $\{(1,1),(2,2),\ldots,(n,n)\}$ is  a kernel too.
\item Any nonempty subset of $\mathsf{G}_n$ could not be a kernel. This can be proved by induction on $n$, since $\mathsf G_n$ can be embedded in $\mathsf G_{n+1}$; compare $\mathsf G_7$ and $\mathsf G_8$ in Fig.~\ref{G_n}.
\end{itemize}
The tedious details are left to the interested reader.
\end{proof}
Recall that a polynomial $f(z_1\ldots,z_m)\in\mathbb{R}[z_1,\ldots,z_m]$ is said to be {\em stable}, if
$f(z_1,\ldots,z_m)\not=0$ whenever $z_1,\ldots,z_m\in\{z\in\mathbb{C}:\ \Im(z)>0\}$. It is well known that the
stability of the multivariate generating functions implies that their univariate counterparts, obtained by diagonalization, have only real zeros. By using the theory of stability, Br{\"a}nd{\'e}n,  Haglund, Visontai and Wagner~\cite{bhvw} proved that $\per(zU_n+A)$ is a polynomial in $z$ with only real zeros if $A$ is a matrix $[a_{i,j}]_{1\ls i,j\ls n}$
with $a_{1,j}\gs a_{2,j}\gs\ldots\gs a_{n,j}$ for all $j=1,\ldots,n$. Can the techniques in their paper be employed to find multivariable extension of the variation of the Eulerian polynomials
$
\sum_{\pi\in\S_{n}}z^{\widetilde\exc(\pi)}
$
that possesses certain nice stable property?

An interpretation for the $\gamma$-coefficients $\gamma_{n,k}^{(2)}$, defined in~\eqref{gamma:2}, of the $2$-Eulerian polynomials  in terms of weakly increasing trees on $\{1,1,2,2,\ldots,n-1,n-1,n\}$  without young leaves has been found in~\cite{LMMZ}. Corollary~\ref{cor:trees} asserts that the diagonal coefficient $\gamma_{n,n-1}^{(2)}$ enumerates alternating multipermutations in $\S_n^{(2)}$, which may shed some light on finding an interpretation for $\gamma_{n,k}^{(2)}$ in terms of certain class of permutations on $\{1,1,2,2,\ldots,n,n\}$.

\section*{Acknowledgements}
The authors would like to thank the anonymous referee for making helpful comments that led to a better presentation of this paper. The first, second and third authors were supported by the National Natural Science Foundation of China
(grants 12171059, 12322115 \& 12271301, and 12371004, respectively). The first author was also supported by the Mathematical Research Center of Chongqing University.

\end{document}